\documentclass[12pt,reqno]{amsart}
\usepackage{amscd,amsmath,amsthm,amssymb}
\usepackage{color}
\usepackage{tikz}
\usepackage{url}
\usepackage{circuitikz}
\usepackage{float}

\usepackage{lipsum}

\newtheorem{Theorem}{Theorem}[section]
\newtheorem{Lemma}[Theorem]{Lemma}
\newtheorem{Corollary}[Theorem]{Corollary}
\newtheorem{Proposition}[Theorem]{Proposition}
\newtheorem{Remark}[Theorem]{Remark}

\newtheorem{Example}[Theorem]{Example}
\newtheorem{Examples}[Theorem]{Examples}
\newtheorem{Definition}[Theorem]{Definition}
\newtheorem{Question}[Theorem]{Question}

\newtheorem{Conjecture}{Conjecture}
\newtheorem{Thm}[Conjecture]{Theorem}

\def\qed{\ifhmode\textqed\fi
	\ifmmode\ifinner\quad\qedsymbol\else\dispqed\fi\fi}
\def\textqed{\unskip\nobreak\penalty50
	\hskip2em\hbox{}\nobreak\hfill\qedsymbol
	\parfillskip=0pt \finalhyphendemerits=0}
\def\dispqed{\rlap{\qquad\qedsymbol}}

\def\ZZ{\mathbb{Z}}
\def\NN{\mathbb{N}}
\def\QQ{\mathbb{Q}}
\def\m{\mathfrak{m}}
\def\n{\mathfrak{n}}
\def\Hom{\operatorname{Hom}}

\def\Ass{\operatorname{Ass}}
\def\Att{\operatorname{Att}}

\def\Ker{\operatorname{Ker}}

\def\depth{\operatorname{depth}}

\def\Im{\operatorname{Im}}

\def\supp{\operatorname{supp}}

\def\v{\operatorname{v}}

\def\ini{\operatorname{in}}
\def\Spec{\operatorname{Spec}}
\def\Proj{\operatorname{Proj}}

\def\Ext{\operatorname{Ext}}

\def\H{\operatorname{H}}
\def\E{\operatorname{E}}
\def\D{\operatorname{D}}
\def\Hilb{\operatorname{Hilb}}

\def\Supp{\operatorname{Supp}}

\def\CoKer{\operatorname{CoKer}}
\def\Tor{\operatorname{Tor}}

\begin{document}
	
	\title{Serre depth and local cohomology}
	\author{Antonino Ficarra}
	
	\address{Antonino Ficarra, BCAM -- Basque Center for Applied Mathematics, Mazarredo 14, 48009 Bilbao, Basque Country -- Spain, Ikerbasque, Basque Foundation for Science, Plaza Euskadi 5, 48009 Bilbao, Basque Country -- Spain}
	\email{aficarra@bcamath.org,\,\,\,\,\,\,\,\,\,\,\,\,\,antficarra@unime.it}
	
	%\thanks{}
	\vspace*{-0.4em}
	\begin{abstract}
		We introduce a fundamental homological invariant, called \emph{Serre depth}, which stratifies Serre's conditions in the same way that depth stratifies the Cohen-Macaulay property. We study the Serre depths of modules over arbitrary Noetherian local rings and over standard graded algebras over a field, extending the polynomial ring case due to Muta and Terai. Under mild hypotheses, we show that the $r$-th Serre depth of a finitely generated module $M$ measures the deviation of $M$ from satisfying Serre's condition $(S_r)$. The main results of the paper can be summarized as follows:
		\begin{enumerate}
			\item[(i)] We establish the basic properties of Serre depth and prove that it is invariant under completion.
			\item[(ii)] If the base ring $R$ is a homomorphic image of a Gorenstein ring, we show that a finitely generated $R$-module $M$ is equidimensional and satisfies $(S_r)$ if and only if its $r$-th Serre depth equals its Krull dimension. Analogous statements are obtained for schemes.
			\item[(iii)] For a homogeneous ideal in a standard graded polynomial ring over a field, we compare its Serre depths with those of its initial ideal.
			\item[(iv)] We characterize the Serre depths of a monomial ideal in terms of its skeletons and prove that the Serre depths of sufficiently large powers of a monomial ideal stabilize; the proof uses Presburger arithmetic.
		\end{enumerate}
	\end{abstract}
	
	\maketitle\vspace*{-1.8em}
	\section*{Introduction}
	
	Among Noetherian local rings, Cohen-Macaulay rings hold a distinguished place, as reflected in current trends in Commutative Algebra. In the words of Melvin Hochster: ``\textit{Life is really worth living in a Cohen-Macaulay ring}".
	
	A Cohen-Macaulay ring is a Noetherian local ring $R$ in which every system of parameters forms a regular sequence. Denote by $\depth(R)$ the depth of $R$. Then $\depth(R)$ is always less than or equal to $\dim(R)$, and equality holds if and only if $R$ is Cohen-Macaulay. Thus, the depth of $R$ measures how far the ring is from being Cohen-Macaulay. Similarly, one can define Cohen-Macaulay modules.
	
	There is a natural refinement of the Cohen-Macaulay property. We say that a $R$-module satisfies \textit{Serre's condition} $(S_r)$ if
	$$
	\depth_{R_P}(M_P)\ \ge\ \min\{r,\,\dim_{R_P}(M_P)\},
	$$
	for all primes $P$ belonging to the support $\Supp_R(M)$.
	
	It is easily seen that $M$ satisfies $(S_1)$ if and only if it does not have embedded associated primes. Furthermore, $M$ is Cohen-Macaulay if and only if it satisfies $(S_r)$ for all $r\ge1$. Serre's conditions $(S_r)$ stratify Noetherian local rings, provide a flexible generalization of the Cohen-Macaulay property, and play a crucial role in the study of singularities and homological behaviour of finitely generated modules.
	
	Recently, inspired by results of Schenzel (see \cite{Sc0,Sc1,Sc2}), Muta and Terai introduced in \cite{MT} (see also \cite{KMT,PPTYa,PPTYb}) the notion of Serre depth of a finitely generated graded module $M$ in the polynomial ring setting as a measure of the deviation of $M$ from satisfying Serre's condition $(S_r)$. The authors extensively studied the Serre depth of squarefree monomial ideals and their symbolic powers.
	
	In this paper, our goal is to extend the notion of Serre depth in the setting of arbitrary Noetherian local rings or standard graded algebras over a field.
	
	Let $(R,\m)$ be a Noetherian local ring with residue class field $R/\m\cong K$. Let $r\ge1$ be a positive integer, and let $M$ be a $R$-module. Let $\E_R(K)$ be an injective hull of $K$, and for a $R$-module $N$ we denote by $\D_R(N)=\Hom_R(N,\E_R(K))$ the Matlis dual of $N$. Then, the \textit{$r$-th Serre depth} of $M$ is defined as
	$$
	S_r\text{-}\!\depth_R(M)\ =\ \min\{j:\ \dim\D_R(\H_{\m}^j(M))\ge j-r+1\}
	$$
	if $M\ne0$; otherwise we set $S_r\text{-}\!\depth_R(0)=-\infty$. Similar definitions can be made in the graded setting. In Proposition \ref{Prop:inequalities} we prove that:
	\begin{enumerate}
		\item[\textup{(a)}] $\dim(M)\ge S_1\text{-}\!\depth(M)\ge\dots\ge S_{r-1}\text{-}\!\depth(M)\ge S_r\text{-}\!\depth(M)\ge\cdots$.\smallskip
		\item[\textup{(b)}] If $M$ is finitely generated, then $S_r\text{-}\!\depth(M)=\depth(M)$, for all positive integers $r\ge\dim(M)$.
	\end{enumerate}
    
    Statement (b) indicates that the Serre depth plays a role analogous to the usual depth. By (a), $S_r\text{-}\!\depth(M)\le\dim(M)$. In Theorem \ref{Thm:general-implication}, we prove that if equality holds, then $M$ is unmixed, hence equidimensional, and satisfies $(S_r)$.
    
    Our main result is:
	
	\begin{Thm}\label{ThmA}
		Let $(R,\m)$ be a Noetherian local ring or a standard graded $K$-algebra. Let $M$ be a finitely generated $R$-module, which we assume is homogeneous if $R$ is a $K$-algebra. Consider the following statements.
		\begin{enumerate}
			\item[\textup{(a)}] $M$ is equidimensional and satisfies Serre's condition $(S_r)$.
			\item[\textup{(b)}] $S_r\text{-}\!\depth(M)=\dim(M)$.
		\end{enumerate}\smallskip
		Then \textup{(b)} $\Rightarrow$ \textup{(a)}, and the converse implication holds if $R$ is a homomorphic image of a Gorenstein ring. Furthermore, in \textup{(a)} the condition that $M$ is equidimensional can be dropped if $r\ge2$, $R$ is a homomorphic image of a Gorenstein ring, and the module $M$ is either indecomposable or a quotient of $R$.
	\end{Thm}

    This result highlights the role of Serre depths as a natural substitute of the usual depth, in measuring deviation from Serre's conditions.
    
    Theorem \ref{ThmA} substantially improves several scattered results in the literature. An earlier version of this result appeared in \cite[Lemma (2.1)]{Sc0}, where Schenzel proved that (a) and (b) are equivalent only when $M=R=S/I$ is a homomorphic image of a Gorenstein ring $S$. This version also appears (without proof) in the book of Vasconcelos \cite[Proposition 3.51]{V} under the extra unnecessary assumption that $R$ is equidimensional, and was recently proved explicitly by Dao, Ma and Varbaro in \cite[Proposition 2.11]{DMV}, but only in the polynomial ring case.
    
    Subsequently, Schenzel extended this result to modules (see \cite[Lemma 3.2.1]{Sc1} and \cite[Lemma 1.9]{Sc2}), again under the assumption that $R=S/I$, and the extra unnecessary assumption that $M$ is equidimensional.
    
    We emphasize that implication (b) $\Rightarrow$ (a) was previously unknown for general rings. Unfortunately, in general the converse is false, as we show in Example \ref{Ex:HRW}. On the other hand, the class of rings that are homomorphic images of a Gorenstein ring is quite large. It includes all complete Noetherian local rings (by Cohen's structure theorem \cite[Theorem 29.4(ii)]{Mat}), all Cohen-Macaulay rings having a canonical module (see \cite[Theorem 3.3.6]{BH}) and all standard graded $K$-algebras. Thus, for all these rings, statements (a) and (b) in Theorem \ref{ThmA} are equivalent.\smallskip
    
    We now present a geometric reformulation of Theorem \ref{ThmA}.
    
    Let $X$ be a Noetherian affine local scheme. That is, $X=\Spec(R)$ is the spectrum of a Noetherian local ring $R$. Let $\mathcal{F}$ be a coherent sheaf on $X$. Recall that for each $x\in X$ (which corresponds to a prime ideal $P\in\Spec(R)$), $\mathcal{F}_x$ is a finitely generated module over the Noetherian local ring $\mathcal{O}_{X,x}=R_P$. Let $r\ge1$ be an integer and fix $x\in X$. We define that \textit{$r$-th Serre depth of $\mathcal{F}$ at $x$} as
    $$
    S_r\text{-}\!\depth_x(\mathcal{F})\ =\ S_r\text{-}\!\depth_{\mathcal{O}_{X,x}}(\mathcal{F}_x).
    $$
    Similarly, we put
    $$
    \dim_x(\mathcal{F})\ =\ \dim_{\mathcal{O}_{X,x}}(\mathcal{F}_{x}),\quad\textup{for all}\ x\in X.
    $$
    
    Analogous definitions can be given for a coherent sheaf over a standard graded projective scheme $X$ over a field $K$. In this case $X=\Proj(R)$ is the homogeneous spectrum of a standard graded $K$-algebra $R=\bigoplus_{k\ge0}R_k$. That is, the set of homogeneous prime ideals $P\in\Spec(R)$ which do not contain the homogeneous maximal ideal $\m=R_+=\bigoplus_{k>0}R_k$.
    
    We say that $X$ is \textit{locally a homomorphic image of a Gorenstein ring} if $\mathcal{O}_{X,x}$ is a homomorphic image of a Gorenstein ring for all $x\in X$. For standard graded projective schemes over a field this is trivially true. We say that a coherent sheaf on $\mathcal{F}$ satisfies Serre's condition $(S_r)$ if $\mathcal{F}_x$ satisfies Serre's condition $(S_r)$ for all $x\in X$. Analogously, we say that $\mathcal{F}$ is equidimensional (respectively, indecomposable) if $\mathcal{F}_x$ is an equidimensional (respectively, indecomposable) $\mathcal{O}_{X,x}$-module for all $x\in X$.
    
    \begin{Thm}\label{ThmB}
    	Let $X$ be either a Noetherian affine local scheme or a standard graded projective scheme over a field $K$. Let $\mathcal{F}$ be a coherent sheaf on $X$. Consider the following statements.
    	\begin{enumerate}
    		\item[\textup{(a)}] $\mathcal{F}$ is equidimensional and satisfies Serre's condition $(S_r)$.
    		\item[\textup{(b)}] $S_r\text{-}\!\depth_x(\mathcal{F})=\dim_x(\mathcal{F})$ for all $x\in X$.
    	\end{enumerate}\smallskip
    	Then \textup{(b)} $\Rightarrow$ \textup{(a)}, and the converse implication holds if $X$ is locally a homomorphic image of a Gorenstein ring. Furthermore, in \textup{(a)} the condition that $\mathcal{F}$ is equidimensional can be dropped if $r\ge2$, $X$ is locally a homomorphic image of a Gorenstein ring, and the sheaf $\mathcal{F}$ is either indecomposable or an ideal sheaf.
    \end{Thm}
    
    Now, we describe the content and main findings of this paper.\smallskip
    
    In Section \ref{sec1}, we introduce the notion of Serre depth (Definition \ref{Def:SerreDepth}). It turns out that $S_r\text{-}\!\depth(M)=\depth(M)$, whenever $M$ is Cohen-Macaulay, or sequentially Cohen-Macaulay (Examples \ref{Ex:(i)-(vi)}). We investigate the behaviour of the Serre depth under specializations, direct sums, and short exact sequences.
    
    A fundamental property, needed in many proofs, is the invariance of the Serre depth under completion.
    \begin{Thm}\label{ThmC}
    	Let $R$ be local. For any finitely generated $R$-module $M$, we have
    	$$
    	S_r\text{-}\!\depth_R(M)\ =\ S_r\text{-}\!\depth_{\widehat{R}}(\,\widehat{\!M\!}\,),
    	$$
    	for all $r\ge1$.
    \end{Thm}

    The proof of this result is given in Section \ref{sec2}. The most technically difficult part of the paper is the proof of Theorem \ref{ThmA} which is given in Section \ref{sec3}. As an outcome of the arguments given there, in Corollary \ref{Cor:Serre} we prove that if $R$ is a homomorphic image of a Gorenstein ring, and $M$ is a finitely generated $R$-module, then
    $$
    S_r\text{-}\!\depth_{R_P}(M_P)\ \ge\ S_{r+\dim(R/P)}\text{-}\!\depth_R(M)-\dim(R/P),
    $$
    for all primes $P\in\Supp_R(M)$ and $r\ge1$. Moreover, equality holds for a given $P\in\Supp_R(M)$ if $S_{r-\dim(R/P)}\text{-}\!\depth_R(M)=\dim_R(M)$.\smallskip
	
	Unfortunately, the implication (a) $\Rightarrow$ (b) in Theorem \ref{ThmA} is not true for general modules (Example \ref{Ex:HRW}). Heinzer, Rotthaus and Wiegand \cite[Theorem 19.11]{HRW} discovered a $2$-dimensional Noetherian local domain $R$ whose completion does not satisfy $(S_1)$. Note that, $R$ is unmixed and satisfies $(S_1)$. However $S_1\text{-}\!\depth_R(R)<2=\dim(R)$.
	
	The equidimensional condition in statement (a) of Theorem \ref{ThmA} can be dropped if $r\ge2$, $R$ is a homomorphic image of a Gorenstein ring and $M$ is either a quotient ring $R/I$ or an indecomposable $R$-module. See the Corollaries \ref{Cor:Serre} and \ref{Cor:Serre1}. We remark that these results are independent statements.\smallskip
	
	In the last two sections, we focus on homogeneous ideals in a standard graded polynomial ring $S=K[x_1,\dots,x_n]$ over a field $K$. In the following result, we compare the Serre depths of $S/I$ with those of the initial degeneration $S/\ini_<(I)$, giving an analogue of the classical upper semicontinuity valid for the depth.
	
	\begin{Thm}\label{ThmD}
		Let $S=K[x_1,\dots,x_n]$ be the standard graded polynomial ring over a field $K$, let $I\subset S$ be a homogeneous ideal, and let $<$ be a monomial order on $S$. Then
		$$
		S_r\text{-}\!\depth(S/\ini_<(I))\ \le\ S_r\text{-}\!\depth(S/I),
		$$
		for all $r\ge1$. Furthermore, equality holds if $\ini_<(I)$ is squarefree.
	\end{Thm}
	
	Theorem \ref{ThmD} then leads to the problem of efficiently computing the Serre depth of a monomial ideal. To this end, we generalize results of Smith \cite{Smith}, Hibi \cite{Hibi}, Herzog-Soleyman Jahan-Zheng \cite{HSZ}, and Muta-Terai \cite{MT}, and prove the following\smallskip
	
	\begin{Thm}\label{ThmE}
		Let $S=K[x_1,\dots,x_n]$ be the standard graded polynomial ring over a field $K$, and let $I\subset S$ be a monomial ideal. Then
		$$
		S_r\text{-}\!\depth(S/I)\ =\ \max\{i:\ S/\Sigma_i^{\bf g}(I)\ \textit{satisfies Serre's condition}\ (S_r)\},
		$$
		for all $r\ge1$ and all ${\bf g}\in\ZZ_{\ge0}^n$ with ${\bf g}\ge{\bf deg}(I)$. Furthermore, $S/\Sigma_{j}^{\bf g}(I)$ satisfies Serre's condition $(S_r)$ for all $j\le S_r\text{-}\!\depth(S/I)$.
	\end{Thm}
	
	For the definitions of ${\bf deg}(I)$ and the skeleton ideals $\Sigma_i^{\bf g}(I)$, see Section \ref{sec4}.\newpage
	
	One of the most striking homological behaviours of the depth was established by Brodmann in 1979 \cite{B79a}. He proved that the function $k\mapsto\depth(R/I^k)$ is constant for all $k\gg0$. See also \cite{Conca,CHT,FS,Kod} for related results. A short proof of this fact was given by Herzog and Hibi \cite[Theorem 1.1]{HH}, who conjectured that the depth function can be as wild as possible: any convergent non-negative numerical function can occur as the depth function $k\mapsto\depth(R/I^k)$ of a homogeneous ideal $I$. This conjecture has been answered affirmatively by H\`a, Nguyen, Trung and Trung \cite{HNTT}.
	
	To prove a similar result for $S_r\text{-}\!\depth(R/I^k)$, a natural approach is to show that
	\begin{equation}\label{eq:dim_R-const}
		\dim_R\D_R(\H_\m^j(R/I^{k+1}))=\dim_R\D_R(\H_\m^j(R/I^{k})),\quad\text{for all}\ j\ge0\ \textup{and}\ k\gg0.
	\end{equation}
	When $R$ is complete, the above modules are finitely generated, and hence equation (\ref{eq:dim_R-const}) would be true if we could prove that
	\begin{equation}\label{eq:ass_R-const}
	\Ass_R(\D_R(\H_\m^j(R/I^{k+1})))=\Ass_R(\D_R(\H_\m^j(R/I^{k}))),\quad\text{for all}\ j\ge0\ \textup{and}\ k\gg0.
    \end{equation}
	Unfortunately, by a famous example of Katzman \cite{Kat} this is not always true (see Example \ref{Ex:notAss}). Nevertheless, equation (\ref{eq:dim_R-const}) may still hold. On the other hand,
	
	\begin{Thm}\label{ThmF}
		Let $S=K[x_1,\dots,x_n]$ be the standard graded polynomial ring over a field $K$, and let $I\subset S$ be a monomial ideal. Then
		$$
		S_r\text{-}\!\depth(S/I^{k+1})\ =\ S_r\text{-}\!\depth(S/I^k),
		$$
		for all $r\ge1$ and all $k\gg0$.
	\end{Thm}
	
	To prove Theorem \ref{ThmF} we use Takayama's formula \cite{T} and Presburger arithmetic \cite{ES}. To this end, we have to carefully analyze the asymptotic vanishing behaviour of the multigraded components of the local cohomology modules.
	
	Presburger arithmetic was used by Dao and Monta\~no in their pioneering paper \cite{DM} and more recently in \cite{DMMNW}. To the best of our knowledge, aside from these works and the present paper, it has rarely appeared in Commutative Algebra. The results in \cite{DM,DMMNW} and Theorem \ref{ThmF} illustrate its potential in the field.
	
	In view of Theorem \ref{ThmF} and Brodmann classical result \cite{B79a}, we expect:
	
	\begin{Conjecture}\label{ConjG}
		Let $(R,\m)$ be a Noetherian local ring or a standard graded $K$-algebra. Let $I\subset R$ be an ideal, which we assume is homogeneous if $R$ is a $K$-algebra. Then
		$$
		S_r\text{-}\!\depth(R/I^{k+1})\ =\ S_r\text{-}\!\depth(R/I^k),
		$$
		for all $r\ge1$ and all $k\gg0$.
	\end{Conjecture}

    We can prove Conjecture \ref{ConjG} in the following cases (see Corollary \ref{Cor:resume}):
    
	\begin{enumerate}
		\item[\textup{(a)}] $R$ is a polynomial ring over a field $K$ and $I$ is a monomial ideal.
		\item[\textup{(b)}] $R$ is Cohen-Macaulay domain with a $($graded, if $R$ is a $K$-algebra$)$ canonical module $\omega_R$ and $\dim(R)\le3$.
		\item[\textup{(c)}] $R/I^k$ is $($sequentially$)$ Cohen-Macaulay for all $k\gg0$.
		\item[\textup{(d)}] $\depth(R/I^k)=0$ for all $k\gg0$.
		\item[\textup{(e)}] $R$ is a homomorphic image of a Gorenstein ring, $\depth(R/I^k)\!=\!\dim(R/I^k)\!-\!1$ for all $k\gg0$ and $r=1$.
	\end{enumerate}

	By Theorem \ref{ThmF}, the function $k\mapsto(S_1\text{-}\!\depth(S/I^k),\dots,S_d\text{-}\!\depth(S/I^k))$, with $I$ a monomial ideal in $S$ and $d=\max\{1,\dim(S/I)\}$, is eventually constant. We call this function, the \textit{depth strata} of $I$. In view of \cite{HNTT}, we expect that any convergent non-negative function $f:k\in\ZZ_{>0}\mapsto f(k)=(a_{1,k},\dots,a_{d,k})\in\ZZ_{\ge0}^d$ with $a_{1,k}\ge\cdots\ge a_{d,k}$ for all $k\ge1$, and if $a_{d,k}=0$ then $a_{i,k}=0$ for all $i$ (see Proposition \ref{Prop:depth0}), is the depth strata of a monomial ideal $I$.
	
	\section{Basic properties of the Serre depth}\label{sec1}
	
	Let $(R,\m)$ be either a Noetherian local ring or a standard graded $K$-algebra with unique homogeneous maximal ideal $\m$ and residue class field $R/\m\cong K$.\smallskip
	
	Recall that the Krull dimension of a non-zero (not necessarily finitely generated) $R$-module $M$ is defined as
	$$
	\dim_R(M)\ =\ \sup\{\dim(R/P):\ P\in\Supp_R(M)\},
	$$
	where $\Supp_R(M)=\{P\in\Spec(R):\ M_P\ne0\}$ is the support of $M$. By convention, the Krull dimension of the zero module is $-\infty$. When the base ring $R$ is clear from the context, we omit the subscript $R$ and we simply write $\dim(M)$ and $\Supp(M)$ instead of $\dim_R(M)$ and $\Supp_R(M)$.\smallskip
	
	When $R$ is a local ring, let $\E_R(K)$ be an injective hull of $K$.
	
	For a $R$-module $M$, which we assume is homogeneous if $R$ is a $K$-algebra, the (\textit{graded}, if $R$ is a $K$-algebra) \textit{Matlis dual} of $M$ is the $R$-module defined as
	$$
	\D_R(M)\ =\ \begin{cases}
		\,\Hom_{R}(M,\E_R(K)),&\textup{if}\ R\ \textup{is local},\\
		\,\Hom_{K}(M,K),&\textup{if}\ R\ \textup{is a}\ K\textup{-algebra}.
	\end{cases}
	$$
	
	When $R$ is a $K$-algebra, and $M$ is $G$-graded, with $G=\ZZ^n$ for some $n\ge1$, then $\D_R(M)$ is also $G$-graded (see \cite[Exercise 14.4.2(i)]{BS}) with ${\bf g}$-th graded piece:
	\begin{equation}\label{eq:gradedPieceMatlis}
		\D_R(M)_{\bf g}\ =\ \Hom_K(M,K)_{\bf g}\ =\ \Hom_K(M_{-{\bf g}},K),\quad\textup{for all}\ {\bf g}\in G.
	\end{equation}
	
	\begin{Definition}\label{Def:SerreDepth}
		\rm Let $r\ge1$, and let $M$ be a $R$-module, which we assume is homogeneous if $R$ is a $K$-algebra. The $r$-th \textit{Serre depth} of $M$ is defined as
		$$
		S_r\text{-}\!\depth_R(M)\ =\ \min\{j:\ \dim\D_R(\H_{\m}^j(M))\ge j-r+1\}
		$$
		if $M\ne0$; otherwise, we set $S_r\text{-}\!\depth_R(0)=-\infty$. When the base ring $R$ is clear from the context, we omit the subscript $R$.
	\end{Definition}\smallskip
    
    Definition \ref{Def:SerreDepth} agrees with the definition given in \cite{MT} when $R=K[x_1,\dots,x_n]$.
    
    \begin{Examples}\label{Ex:(i)-(vi)}
    	\rm (i) Let $(R,\m)$ be local and let $M$ be a $R$-module of finite length. Then $\Supp(M)=\{\m\}$, $\H_\m^0(M)\cong M$ and $\H_\m^j(M)=0$ for all $j>0$. Since $M$ is finitely generated and $\Ass(\E_R(K))=\{\m\}$, we have $$\Ass(\D_R(\H_\m^0(M)))=\Ass(\Hom_R(M,\E_R(K)))=\Supp(M)\cap\Ass(\E_R(K))=\{\m\}.$$
    	Hence, $\dim\D(\H_\m^0(M))=0$ and $\dim\D(\H_\m^j(M))=-\infty$ for $j>0$. We conclude that $S_r\text{-}\!\depth(M)=0$ for all $r\ge1$. Note that $M$ is trivially Cohen-Macaulay.\smallskip
    	
    	(ii) More generally, if $M$ is a finitely generated, Cohen-Macaulay $R$-module, Proposition \ref{Prop:inequalities} implies that $S_r\text{-}\!\depth(M)=\depth(M)=\dim(M)$, for all $r\ge1$.\smallskip
    	
    	(iii) Let $R$ be Cohen-Macaulay with canonical module $\omega_R$ (which we assume is graded, if $R$ is a $K$-algebra). Following Stanley \cite{St} (see also \cite{Sc}), a finitely generated $R$-module $M$, which we assume is homogeneous if $R$ is a $K$-algebra, is called \textit{sequentially Cohen-Macaulay} if there exists a filtration $$0=M_0\subset M_1\subset\cdots\subset M_t=M$$ of $M$, where the modules $M_i$ are homogeneous if $R$ is a $K$-algebra, such that
    	\begin{enumerate}
    		\item[(a)] $M_{i}/M_{i-1}$ is a Cohen-Macaulay module, for all $i=1,\dots,t$, and
    		\item[(b)] $\dim(M_1/M_0)<\dim(M_2/M_1)<\dots<\dim(M_t/M_{t-1})$.
    	\end{enumerate}
        
        By Peskine characterization (see \cite[Theorem 1.4]{HS}), and Grothendieck's local duality (see \cite[Theorem 3.5.8]{BH} and \cite[Theorem 3.6.19]{BH} for the graded version), the following conditions are equivalent:
        \begin{enumerate}
        	\item[(a)] $M$ is sequentially Cohen-Macaulay.
        	\item[(b)] $\D(\H_\m^j(M))\cong\Ext^{\dim(R)-j}_R(M,\omega_R)$ is either $0$ or a Cohen-Macaulay module of dimension $j$, for all $j=0,\dots,\dim(M)$.
        \end{enumerate}
        
        Let $M$ be sequentially Cohen-Macaulay with $\depth(M)=t$. By \cite[Theorem 3.5.7]{BH} we have $t=\inf\{j:\H_\m^j(M)\ne0\}$. Part (b) implies that $\dim\D(\H_\m^t(M))=t\ge t-r+1$ for all $r\ge1$. It follows that $S_r\text{-}\!\depth(M)=\depth(M)$ for all $r\ge1$.\smallskip
        
    	(iv) If $R$ is a $1$-dimensional Cohen-Macaulay ring, and $I\subset R$ is a non-zero ideal, then $S_r\text{-}\!\depth(R/I)=S_r\text{-}\!\depth(I)-1=0$ for all $r\ge1$. In fact $\dim(R/I)=0$ and so $S_r\text{-}\!\depth(R/I)=0$ by Proposition \ref{Prop:inequalities}(a). On the other hand, $I$ is automatically Cohen-Macaulay of dimension $1$, and so $S_r\text{-}\!\depth(I)=1$ for all $r\ge1$.\smallskip
    	
    	(v) Let $R=K[x,y,z]$ and $I=(x^2-y^2,x^3,yz^2,z^3)$. Using Macaulay2 \cite{GDS} and Grothendieck's local duality, we have $\dim\D(\H_\m^1(I))=\dim\Ext^2_R(I,\omega_R)=0$, $\dim\D(\H_\m^3(I))=\dim\Ext^0_R(I,\omega_R)=3$ and $\dim\D(\H_\m^j(I))=-\infty$ for the other values of $j$. It follows that $S_1\text{-}\!\depth(I)=3$ and $S_r\text{-}\!\depth(I)=1$ for all $r\ge2$.\smallskip
    	
    	(vi) Let $R=K[x,y,z]$, $I=(x^2-y^2,x^3,yz^2,z^3)$ and $M=R/I$. The computations in (v) and Proposition \ref{Prop:Sr-I-R/I} yield that $S_r\text{-}\!\depth(M)=\depth(M)=0$, for all $r\ge1$.
    \end{Examples}\smallskip

    If $(R,\m)$ is local, the \textit{$\m$-adic completion} of $R$ is defined as inverse limit:
    $$
    \widehat{R}\ =\ \lim\limits_{\substack{\longleftarrow\\ j}}R/\m^jR.
    $$
    
    Similarly, $\widehat{\!M\!}=\lim\limits_{\longleftarrow}{\phantom{f}\!\!\!\!}_jM/\m^jM$ is called the $\m$-adic completion of a $R$-module $M$.
    
    We say that $R$ (respectively, $M$) is \textit{$\m$-adically complete} if $R\cong\widehat{R}$ (respectively, $M\cong\widehat{\!M\!}$\,). By \cite[Theorem 8.7]{Mat}, if $M$ is finitely generated we have $M\otimes_R\widehat{R}\cong\widehat{\!M\!}$.\smallskip
    
    \begin{Theorem}\label{Thm:completion}
    	Let $R$ be local. For any finitely generated $R$-module $M$, we have
    	$$
    	S_r\text{-}\!\depth_R(M)\ =\ S_r\text{-}\!\depth_{\widehat{R}}(\,\widehat{\!M\!}\,),\quad\text{for all}\ r\ge1.
    	$$
    \end{Theorem}

    We postpone the proof of this result until Section \ref{sec2}. We mention it here because it is needed for the proof of the following proposition, which collects some basic properties of the Serre depth.
    
    \begin{Proposition}\label{Prop:inequalities}
    	Let $M$ be a $R$-module, which we assume is homogeneous if $R$ is a $K$-algebra. The following statements hold.
    	\begin{enumerate}
    		\item[\textup{(a)}] $\dim(M)\ge S_1\text{-}\!\depth(M)\ge\dots\ge S_{r-1}\text{-}\!\depth(M)\ge S_r\text{-}\!\depth(M)\ge\cdots$.
    		\item[\textup{(b)}] If $M$ is finitely generated, then $S_r\text{-}\!\depth(M)=\depth(M)$, for all positive integers $r\ge\dim(M)$.
    	\end{enumerate}
    \end{Proposition}
    \begin{proof}
    	(a) By definition it is clear that $S_{r}\text{-}\!\depth(M)\ge S_{r+1}\text{-}\!\depth(M)$ for all $r\ge1$. By Grothendieck's vanishing theorem \cite[Theorem 6.1.2]{BS}, we have $\H_\m^j(M)=0$ for all $j>\dim(M)$. Hence $\D(\H_\m^j(M))=0$ for all $j>\dim(M)$. This shows that $\dim(M)\ge S_{r}\text{-}\!\depth(M)$, for all $r\ge1$.\smallskip
    	
    	(b) Let $M$ be finitely generated and $d=\dim(M)$. If $d=0$, then $\depth(M)=0$ as well, and by part (a) we have $S_r\text{-}\!\depth(M)=0$ for all $r\ge1$. So, we may assume that $d\ge1$. It is well-known (see \cite[Theorem 3.5.7]{BH}) that
    	$$
    	\depth(M)\ =\ \inf\{j:\ \H_\m^j(M)\ne0\}.
    	$$
    	
    	Let $t=\depth(M)$. If $j<t$, then $\D(\H_\m^j(M))=0$, hence $\dim\D(\H^j_\m(M))=-\infty$ and so $S_r\text{-}\!\depth(M)\ge t$ for all $r\ge d$. It remains to prove that $\dim\D(\H_\m^t(M))\ge t-r+1$, for all $r\ge d$. This will imply that $S_r\text{-}\!\depth(M)=t$ for all $r\ge d$. We distinguish the two possible cases.\smallskip
    	
    	\textbf{Case 1.} Let $t<r$. Then $t-r+1\le0$. Since $\H_\m^t(M)\ne0$, we have $\D(\H_\m^t(M))\ne0$ as well, and so $\dim\D(\H_\m^t(M))\ge0\ge t-d+1$, as desired.\smallskip
    	
    	\textbf{Case 2.} Let $t=r$. Since $t\le d$ and $r\ge d$, it follows that $r=d$. Hence $t-r+1=1$ and $M$ is Cohen-Macaulay of dimension $d\ge1$. We claim that $\dim\D(\H_\m^t(M))=d$ and this will conclude the proof.
    	
    	If $R$ is a local ring, we may assume it is complete. Indeed, by \cite[Corollary 2.1.8]{BH}, we have $\dim_R(M)=\dim_{\widehat{R}}(\,\widehat{\!M\!}\,)$ and $\depth_R(M)=\depth_{\widehat{R}}(\,\widehat{\!M\!}\,)$, and by Theorem \ref{Thm:completion}, we have $S_r\text{-}\!\depth_R(M)=S_r\text{-}\!\depth_{\widehat{R}}(\,\widehat{\!M\!}\,)$ for all $r\ge1$.\smallskip
    	
    	We can write $R\cong S/I$, where $(S,\n)$ is a Gorenstein local ring if $R$ is local or $S=K[x_1,\dots,x_n]$ if $R$ is a $K$-algebra, and $I\subset S$ is an ideal, which is homogeneous if $R$ is a $K$-algebra. This is clear if $R$ is a $K$-algebra. If $R$ is local, since we assumed it is complete, our claim follows by Cohen's structure theorem (\cite[Theorem 29.4(ii)]{Mat}), because regular rings are Gorenstein (\cite[Proposition 3.1.20]{BH}).
    	
    	Having acquired our claim, in the local case \cite[Exercise 3.5.14]{BH} gives
    	$$
    	\D_R(\H_\m^t(M))\ \cong\ \D_S(\H_\n^t(M))\ \cong\ \Ext^{\dim(S)-t}_S(M,\omega_S)
    	$$
    	By \cite[Theorem 3.3.10(i)]{BH}, this $S$-module has Krull dimension $t=d\ge1$, because both $S$ and $M$ are Cohen-Macaulay and $M$ has Krull dimension $d$. Since $I$ annihilates $\D_R(\H_\m^t(M))$, we have $\dim_R\D_R(\H_\m^t(M))=\dim_S\D_R(\H_\m^t(M))=d\ge1$ as well.
    	
    	If $R$ is a $K$-algebra, since we trivially have $\H_\m^t(M)\cong\H_\n^t(M)$, then
    	$$
    	\D_R(\H_\m^t(M))\ =\ \Hom_K(\H_\m^t(M),K)\ \cong\ \Hom_K(\H_\n^t(M),K)\ =\ \D_S(\H_\n^t(M)).
    	$$
    	Finally, using that $\dim(M)=d$ and both $S$ and $M$ are Cohen-Macaulay, by the graded version of \cite[Theorem 3.3.10(i)]{BH}, $\D_R(\H_\m^t(M))$ has Krull dimension $t=d\ge1$ as a $S$-module. As explained before, $\dim_R\D_R(\H_\m^t(M))=t=d$ as well.
    \end{proof}

    \begin{Remark}
    	\rm Alternatively, we may conclude the proof of the second case as follows. By \cite[Theorem 7.3.2]{BS}, since $M$ is Cohen-Macaulay hence unmixed of dimension $d=t\ge1$, we have
    	\begin{equation}\label{eq:att1}
    		\Att(\H_\m^d(M))\ =\ \{P\in\Ass(M):\ \dim(R/P)=d\}\ =\ \Ass(M).
    	\end{equation}
    	Here $\Att(N)$ is the set of \textit{attached primes} of a $R$-module $N$ (see \cite[Definition 7.2.4]{BS}).
    	
    	By \cite[Proposition 3.5.4(a)]{BH}, $\H_\m^d(M)$ is Artinian. Since $R$ is complete, Matlis duality \cite[Theorem 10.2.12(ii)]{BS} (see \cite[Theorem 3.6.17]{BH} for the graded version) implies that $\D(\H_\m^d(M))$ is a Noetherian $R$-module and $\D(\D(\H_\m^d(M)))\cong\H_\m^d(M)$. Hence, \cite[Corollary 10.2.20]{BS} implies that
    	\begin{equation}\label{eq:att2}
    		\Att(\H_\m^d(M))\ =\ \Att(\D(\D(\H_\m^d(M))))\ =\ \Ass(\D(\H_\m^d(M))).
    	\end{equation}
    	Combining equations (\ref{eq:att1}) and (\ref{eq:att2}) we then conclude that $\Ass(\D(\H_\m^d(M)))=\Ass(M)$. Hence $\dim\D(\H_\m^d(M))=\dim(M)=d\ge1$.
    \end{Remark}

    The Serre depth behaves nicely with respect to direct sums.
    
    \begin{Proposition}
    	Let $\{M_i\}_{i\in\mathcal{I}}$ be a family of $R$-modules, which we assume are homogeneous if $R$ is a $K$-algebra. Then
    	$$
    	S_r\text{-}\!\depth(\bigoplus_{i\in\mathcal{I}}M_i)\ =\ \inf_{i\in\mathcal{I}}\{S_r\text{-}\!\depth(M_i)\},\quad\textit{for all}\ r\ge1.
    	$$
    \end{Proposition}
    \begin{proof}
    	Since local cohomology commutes with direct limits (\cite[Theorem 3.4.10]{BS}), hence direct sums, we have $\H_\m^j(\bigoplus_{i\in\mathcal{I}}M_i)\cong\bigoplus_{i\in\mathcal{I}}\H_\m^j(M_i)$ for all $j\ge0$. Combining this isomorphism with \cite[Theorem 6.45(ii)]{Rot}, in the local case we have
    	\begin{align*}
    		\D(\H_\m^j(\bigoplus_{i\in\mathcal{I}}M_i))\ &=\ \Hom_R(\H_\m^j(\bigoplus_{i\in\mathcal{I}}M_i),\E_R(K))\\
    		&\cong\ \Hom_R(\bigoplus_{i\in\mathcal{I}}\H_\m^j(M_i),\E_R(K))\\
    		&\cong\ \prod_{i\in\mathcal{I}}\Hom_R(\H_\m^j(M_i),\E_R(K))\\
    		&\cong\ \prod_{i\in\mathcal{I}}\D(\H_\m^j(M_i)).
    	\end{align*}
    
    	When $R$ is a $K$-algebra, replacing $\Hom_R(\underline{\phantom{a}},\E_R(K))$ with $\Hom_K(\underline{\phantom{a}},K)$, the isomorphism $\D(\H_\m^j(\bigoplus_{i\in\mathcal{I}}M_i))\cong\prod_{i\in\mathcal{I}}\D(\H_\m^j(M_i))$ holds as well.
    	
    	We claim that the support of a direct product $\prod_{\ell\in\mathcal{L}}N_\ell$ of $R$-modules is the union of the supports of the $R$-modules $N_\ell$. In fact, for all $P\in\Spec(R)$ we have an injective map $(\prod_{\ell\in\mathcal{L}}N_\ell)_P\rightarrow\prod_{\ell\in\mathcal{L}}(N_\ell)_P$. If this map is non-zero, then $(N_h)_P\ne0$ for some $h$. This shows that $\Supp(\prod_{\ell\in\mathcal{L}}N_\ell)\subset\bigcup_{\ell\in\mathcal{L}}\Supp(N_\ell)$. Conversely, given $P\in\bigcup_{\ell\in\mathcal{L}}\Supp(N_\ell)$, we have $(N_{h})_P\ne0$ for some $h$. Since the projection map $\prod_{\ell\in\mathcal{L}}N_\ell\rightarrow N_h$ is surjective, it follows that $P\in\Supp(\prod_{\ell\in\mathcal{L}}N_\ell)$, as desired.
    	
    	Having acquired our claim, the previous isomorphism yields that
    	$$
    	\dim\D(\H_\m^j(\bigoplus_{i\in\mathcal{I}}M_i))\ =\ \dim(\prod_{i\in\mathcal{I}}\D(\H_\m^j(M_i)))\ =\ \sup_{i\in\mathcal{I}}\{\dim\D(\H_\m^j(M_i))\}.
    	$$
    	Our discussion thus far shows that $S_r\text{-}\!\depth(\bigoplus_{i\in\mathcal{I}}M_i)$ is the least integer $j$ such that $\sup_{i\in\mathcal{I}}\{\dim\D(\H_\m^j(M_i))\}\ge j-r+1$. Let $s_i=S_r\text{-}\!\depth(M_i)$ for all $i\in\mathcal{I}$, and set $s=\inf_{i\in\mathcal{I}}\{s_i\}$. By definition of $s$, for all $j<s$ we have $\dim\D(\H_\m^j(M_i))<j-r+1$, and thus $\dim\D(\H_\m^j(\bigoplus_{i\in\mathcal{I}}M_i))<j-r+1$, too. Since $0\le s_i\le\dim(M_i)\le\dim(R)$ for all $i\in\mathcal{I}$, we have $s=s_p$ for some $p$. Hence
    	$$
    	\dim\D(\H_\m^s(\bigoplus_{i\in\mathcal{I}}M_i))\ =\ \sup_{i\in\mathcal{I}}\{\dim\D(\H_\m^s(M_i))\}\ \ge\ \dim\D(\H_\m^{s_p}(M_p))\ \ge\ s-r+1.
    	$$
    	We conclude that $S_r\text{-}\!\depth(\bigoplus_{i\in\mathcal{I}}M_i)=s$, as desired.
    \end{proof}

    Next, we analyze the behaviour of the Serre depth along short exact sequences.
    \begin{Proposition}\label{prop:shortexact}
    	Let
    	$$
    	0\rightarrow U\rightarrow M\rightarrow N\rightarrow 0
    	$$
    	be a short exact sequence of $R$-modules, which we assume are homogeneous if $R$ is a $K$-algebra. Then,
    	\begin{align}
    		\label{eq:Sr-ses1}S_r\text{-}\!\depth(M)\ &\ge\ \min\{S_r\text{-}\!\depth(U),\,S_r\text{-}\!\depth(N)\},\\
    		\label{eq:Sr-ses2}S_r\text{-}\!\depth(N)\ &\ge\ \min\{S_{r+1}\text{-}\!\depth(U)-1,\,S_r\text{-}\!\depth(M)\},\\
    		\label{eq:Sr-ses3}S_r\text{-}\!\depth(U)\ &\ge\ \min\{S_r\text{-}\!\depth(M),\,S_{r-1}\text{-}\!\depth(N)+1\},
    	\end{align}
        where in \textup{(\ref{eq:Sr-ses3})} we assume that $r\ge2$. Furthermore, the following statements hold.\smallskip
        \begin{enumerate}
        	\item[\textup{(a)}] If $S_r\text{-}\!\depth(U)<S_r\text{-}\!\depth(N)$, then $S_r\text{-}\!\depth(U)=S_r\text{-}\!\depth(M)$.\smallskip
        	\item[\textup{(b)}] If $S_r\text{-}\!\depth(M)<S_{r+1}\text{-}\!\depth(U)-1$, then $S_r\text{-}\!\depth(M)=S_r\text{-}\!\depth(N)$.\smallskip
        	\item[\textup{(c)}] If $S_{r-1}\text{-}\!\depth(N)+1<S_r\text{-}\!\depth(M)$ and $r\ge2$, then
        	$$
        	S_r\text{-}\!\depth(U)\ge S_{r-1}\text{-}\!\depth(N)+1\ge S_r\text{-}\!\depth(N)+1\ge S_{r+1}\text{-}\!\depth(U).
        	$$
        \end{enumerate}
    \end{Proposition}
    \begin{proof}
    	The given short exact sequence induces a long exact sequence in local cohomology. Applying the exact contravariant functor $\D(\underline{\phantom{a}})$ (see \cite[10.2.1]{BS} in the local case, and \cite[Theorem 3.6.17]{BH} for the graded version) to this long exact sequence, we obtain the long exact sequence
    	$$
    	\cdots\rightarrow\D(\H_\m^{j+1}(U))\rightarrow\D(\H_\m^j(N))\rightarrow\D(\H_\m^j(M))\rightarrow\D(\H_\m^j(U))\rightarrow\D(\H_\m^{j-1}(N))\rightarrow\cdots.
    	$$
    	
    	For a $R$-module $X$ and integers $j\ge0$ and $r\ge1$, we put $d_X^j=\dim\D(\H_\m^j(X))$ and $s_X^r=S_r\text{-}\!\depth(X)$. From the above long exact sequence, we easily derive:
    	\begin{enumerate}
    		\item[(i)] $d_M^j\le\max\{d_U^j,\,d_N^j\}$,
    		\item[(ii)] $d_N^j\le\max\{d_U^{j+1}\!,\,d_M^j\}$,
    		\item[(iii)] $d_U^j\le\max\{d_M^j,\,d_N^{j-1}\}$.
    	\end{enumerate}\smallskip
    
        Now, we establish the inequalities (\ref{eq:Sr-ses1}), (\ref{eq:Sr-ses2}), (\ref{eq:Sr-ses3}), and the statement (a), (b), (c).\smallskip
        
        Let $s=s_M^r$. Then, by (i) $s-r+1\le d_M^s\le\max\{d_U^s,d_N^s\}$. Hence $s_X^r\le s$ for some $X\in\{U,N\}$. In particular, $s_M^r=s\ge\min\{s_U^r,s_N^r\}$ and  (\ref{eq:Sr-ses1}) holds.\smallskip
        
        Let $s=s_N^r$. By (ii), $s-r+1\le d_N^s\le\max\{d_U^{s+1},d_M^s\}$. So, either $d_U^{s+1}\ge(s+1)-(r+1)+1$ or $d_M^s\ge s-r+1$. Hence, either $s_U^{r+1}\le s+1$ or $s_M^r\le s$. Altogether, $s_N^r=s\ge\min\{s_U^{r+1}-1,s_M^r\}$ and (\ref{eq:Sr-ses2}) holds.\smallskip
        
        Now, let $r\ge2$ and put $s=s_U^r$. By (iii), $s-r+1\le d_U^s\le\max\{d_M^s,d_N^{s-1}\}$. Hence, either $d_M^s\ge s-r+1$ or $d_N^{s-1}\ge(s-1)-(r-1)+1$. Thus, either $s_M^r\le s$ or $s_N^{r-1}\le s-1$. This shows that $s_U^r=s\ge\min\{s_M^r,s_N^{r-1}+1\}$ and (\ref{eq:Sr-ses3}) holds.\smallskip
        
        (a) Assume that $s_U^r<s_N^r$. Inequality (\ref{eq:Sr-ses1}) implies that $s_M^r\ge s_U^r$. By (\ref{eq:Sr-ses3}) we have $s_U^r\ge\min\{s_M^r,s_N^{r-1}+1\}$. Suppose for a moment that $s_U^r\ge s_N^{r-1}+1$. By Proposition \ref{Prop:inequalities}(a) we have $s_N^{r-1}+1>s_N^r$, and so $s_U^r>s_N^r$, against the assumption. Hence, we must have $s_U^r\ge s_M^r$, which together with the previous inequality gives $s_U^r=s_M^r$.\smallskip
        
        (b) Assume that $s_M^r<s^{r+1}_U-1$. By (\ref{eq:Sr-ses2}) we have $s_N^r\ge s_M^r$. Now, by (\ref{eq:Sr-ses1}) we have $s_M^r\ge\min\{s_U^r,s_N^r\}$. Suppose for a moment that $s_M^r\ge s_U^r$. By Proposition \ref{Prop:inequalities}(a) we have $s_U^r>s_U^{r+1}-1$, and so $s_M^r>s_U^{r+1}-1$, against our assumption. Hence, $s_M^r\ge s_N^r$, and so $s_M^r=s_N^r$.\smallskip
        
        (c) Finally, assume that $r\ge2$ and $s_{N}^{r-1}+1<s_M^r$. Then (\ref{eq:Sr-ses3}) yields $s_U^r\ge s_N^{r-1}+1$. Next, (\ref{eq:Sr-ses2}) gives that $s_N^r\ge\min\{s_U^{r+1}-1,s_M^r\}$. Suppose for a moment that $s_N^r\ge s_M^r$. Since by Proposition \ref{Prop:inequalities}(a) we have $s_N^{r-1}+1>s_N^r$, it follows that $s_{N}^{r-1}+1>s_M^r$, against the assumption. Hence $s_N^r\ge s_U^{r+1}-1$. Applying Proposition \ref{Prop:inequalities}(a), we then have $s_U^r\ge s_{N}^{r-1}+1\ge s_N^{r}+1\ge s_U^{r+1}$, as desired.
    \end{proof}
    
    As a consequence of this result, we have
    \begin{Corollary}
    	Let $M$ be a $R$-module, which we assume is homogeneous if $R$ is a $K$-algebra. Let ${\bf x}=x_1,\dots,x_\ell$ be a $M$-sequence, which we assume is homogeneous if $R$ is a $K$-algebra. Then
    	$$
    	S_{r+\ell}\text{-}\!\depth(M)-\ell\ \le\ S_r\text{-}\!\depth(M/({\bf x})M)\ \le\ S_r\text{-}\!\depth(M),
    	$$
    	for all $r\ge1$.
    \end{Corollary}
    \begin{proof}
    	We proceed by induction on $\ell\ge1$. It is enough to consider the case $\ell=1$. So let $x\in\m$ be a $M$-regular element, which we assume is homogeneous if $R$ is a $K$-algebra. Then, multiplication by $x$ is an injective (homogeneous, in the graded setting) map, and so we have the short exact sequence
    	\begin{equation}\label{eq:multiplication}
    	0\rightarrow M\xrightarrow{\cdot x}M\rightarrow M/(x)M\rightarrow0
        \end{equation}
    	which induces the long exact sequence
    	$$
    	\H_\m^j(M)\xrightarrow{\alpha}\H_\m^j(M)\xrightarrow{\beta}\H_\m^j(M/(x)M)\xrightarrow{\gamma}\H_\m^{j+1}(M)\xrightarrow{\delta}\H_\m^{j+1}(M).
    	$$
    	
    	We claim that we have a short exact sequence
    	\begin{equation}\label{eq:particularsequence}
    	0\rightarrow\CoKer(\alpha)\xrightarrow{f}\H_\m^j(M/(x)M)\xrightarrow{g}\Ker(\delta)\rightarrow0,
    	\end{equation}
    	with $f$ defined by $f(a+\Im(\alpha))=\beta(a)$, for all $a+\Im(\alpha)\in\CoKer(\alpha)=\H_\m^j(M)/\Im(\alpha)$, and $g$ defined by $g(b)=\gamma(b)$, for all $b\in\H_\m^j(M/(x)M)$.
    	
    	First, we note that $f$ is well-defined. In fact, if $a+\Im(\alpha)=a'+\Im(\alpha)$, then $a-a'\in\Im(\alpha)=\Ker(\beta)$, then $\beta(a)=\beta(a')$, and so $f(a+\Im(\alpha))=f(a'+\Im(\alpha))$. It is clear that $f$ is a homomorphism. Moreover, if $f(a+\Im(\alpha))=\beta(a)=0$, then $a\in\Ker(\beta)=\Im(\alpha)$, so that $f$ is injective. By the definition of $f$ we have $\Im(f)=\Im(\beta)=\Ker(\gamma)=\Ker(g)$. Finally, $g$ is surjective, because $\Ker(\delta)=\Im(\gamma)$.
    	
    	Hence, applying the functor $\D(\underline{\phantom{a}})$ to (\ref{eq:particularsequence}), we obtain the short exact sequence
    	$$
    	0\rightarrow \D(\Ker(\delta))\xrightarrow{\D(g)}\D(\H_\m^j(M/(x)M))\xrightarrow{\D(f)}\D(\CoKer(\alpha))\rightarrow0,
    	$$
    	which becomes
    	$$
    	0\rightarrow\CoKer(\D(\delta))\xrightarrow{\D(g)}\D(\H_\m^j(M/(x)M))\xrightarrow{\D(f)}\Ker(\D(\alpha))\rightarrow0.
    	$$
    	It follows that
    	\begin{align*}
    		\dim\D(\H_\m^j(M/(x)M))\ &=\ \max\{\dim\CoKer(\D(\delta)),\dim\Ker(\D(\alpha))\}\\
    		&\le\ \max\{\dim\D(\H_\m^j(M)),\dim\D(\H_\m^{j+1}(M))\},
    	\end{align*}
    	because the Krull dimension of a module is always greater or equal than the Krull dimensions of its submodules and its quotients. It follows immediately that
    	$$
    	S_r\text{-}\!\depth(M/(x)M)\ \le\ \min\{S_r\text{-}\!\depth(M),S_{r-1}\text{-}\!\depth(M)\}\ =\ S_{r}\text{-}\!\depth(M),
    	$$
    	where in the last equality we used Proposition \ref{Prop:inequalities}(a). On the other hand, applying Proposition \ref{prop:shortexact} to the short exact sequence (\ref{eq:multiplication}), we have
    	\begin{align*}
    		S_r\text{-}\!\depth(M/(x)M)\ &\ge\ \min\{S_r\text{-}\!\depth(M),S_{r+1}\text{-}\!\depth(M)-1\}\\
    		&=\ S_{r+1}\text{-}\!\depth(M)-1,
    	\end{align*}
    	where in the second step we used Proposition \ref{Prop:inequalities}(a). The proof is complete.
    \end{proof}

	If we assume that $R$ is Cohen-Macaulay, then there is a tight relation between the Serre depths of an ideal $I\subset R$ and its quotient ring $R/I$.
	\begin{Proposition}\label{Prop:Sr-I-R/I}
		Let $R$ be Cohen-Macaulay with $\dim(R)>0$, and let $I\subset R$ be a non-zero proper ideal, which we assume is homogeneous if $R$ is a $K$-algebra. Then
		$$
		S_r\text{-}\!\depth_R(R/I)+1\ =\ S_{r+1}\text{-}\!\depth_R(I),\quad\text{for all}\ r\ge1.
		$$
	\end{Proposition}
    \begin{proof}
    	By Example \ref{Ex:(i)-(vi)}(iv), the statement is immediate when $\dim(R)=1$. So, we assume $\dim(R)\ge2$. Since $R$ is Cohen-Macaulay, \cite[Theorem 3.5.7]{BH} implies that $\H_\m^j(R)=0$ for all $0\le j<\dim(R)$. Hence, the long exact sequence in local cohomology induced by the short exact sequence $0\rightarrow I\rightarrow R\rightarrow R/I\rightarrow0$ implies that
    	\begin{equation}\label{eq:iso-R/I-I}
    		\H_\m^j(R/I)\ \cong\ \H_\m^{j+1}(I),\quad\text{for all}\ j=0,\dots,\dim(R)-2.
    	\end{equation}
    
    	Let $r\ge1$ be an integer and put $s=S_{r+1}\text{-}\!\depth_R(I)$. We distinguish two cases.\smallskip
    	
    	\textbf{Case 1.} Suppose that $s\le\dim(R)-1$. The definition of Serre depth implies that $\dim_R\D_R(\H_\m^j(I))<j-(r+1)+1$ for all $j<s$ and $\dim_R\D_R(\H_\m^s(I))\ge s-(r+1)+1$. The isomorphisms (\ref{eq:iso-R/I-I}) imply that $\dim_R\D_R(\H_\m^j(R/I))<(j+1)-(r+1)+1=j-r+1$ for all $j<s-1$ and $\dim_R\D_R(\H_\m^{s-1}(R/I))\ge s-(r+1)+1=(s-1)-r+1$. This shows that $S_{r}\text{-}\!\depth_R(R/I)=s-1$, as wanted.\smallskip
    	
    	\textbf{Case 2.} Suppose now that $s=\dim(R)$. Then $\dim_R\D_R(\H_\m^j(I))<j-(r+1)+1$ for all $j<\dim(R)$. This implies that $\dim_R\D_R(\H_\m^j(R/I))<(j+1)-(r+1)+1=j-r+1$ for all $j<\dim(R)-1$. Hence $S_r\text{-}\!\depth_R(R/I)\ge\dim(R)-1$. On the other hand, by Proposition \ref{Prop:inequalities}(a) we always have $S_r\text{-}\!\depth_R(R/I)\le\dim_R(R/I)$. Since $I$ is a non-zero ideal, we have $\dim_R(R/I)\le\dim(R)-1$. Hence $S_r\text{-}\!\depth_R(R/I)\le\dim(R)-1$. The desired equality follows.
    \end{proof}

    The previous result does not hold if $R$ is not Cohen-Macaulay.
    \begin{Example}
    	\rm Let $R=K[x,y]/(x^2,xy)$ be a standard graded $K$-algebra. Then $R$ is not Cohen-Macaulay, because $\depth(R)=0<1=\dim(R)$. Let $\m=(x,y)/(x^2,xy)$ be the maximal ideal of $R$. Since $R/\m\cong K$ is Cohen-Macaulay of dimension $0$, we have $S_r\text{-}\!\depth(R/\m)=0$ for all $r\ge1$. Moreover,
    	$$
    	\H_\m^0(\m)\ =\ \Gamma_\m(\m)\ =\ \bigcup_{k\ge0}(0:_{\m}\m^k)\ =\ (x)/(x^2,xy)\ \cong\ K
    	$$
    	and so
    	$\D(\H_\m^0(\m))=\Hom_K(\H_\m^0(\m),K)\cong\Hom_K(K,K)\cong K$. Since this module has finite length, we have $\dim\D(\H_\m^0(\m))=0$ and so $S_r\text{-}\!\depth(\m)=0$ for all $r\ge1$. Finally, $S_r\text{-}\!\depth(R/\m)+1\ne S_{r+1}\text{-}\!\depth(\m)$ for all $r\ge1$.
    \end{Example}
	
	We close this section with a useful equivalent definition of Serre depth in the case that $R$ is a homomorphic image of a Gorenstein ring.
	
	\begin{Proposition}
		Assume that $R=S/I$ is a homomorphic image of a Gorenstein ring $S$. Let $M$ be a finitely generated $R$-module, which we assume is homogeneous if $R$ is a $K$-algebra. Then
		$$
		S_r\text{-}\!\depth_R(M)\ =\ \dim(S)-\max\{i:\ \dim_R\Ext^i_S(M,S)\ge\dim(S)-i-r+1\},
		$$
		for all $r\ge1$.
	\end{Proposition}
	\begin{proof}
		Let $n=\dim(S)$. In the graded case, the assumption is vacuous. In fact, $R\cong S/I$ where $S=K[x_1,\dots,x_n]$ is the standard graded polynomial ring (which is Gorenstein) with homogeneous maximal ideal $\n=(x_1,\dots,x_n)$, and $I\subset S$ is a homogeneous ideal. We trivially have $\H_\m^j(M)\cong\H_\n^j(M)$. Hence
		$$
		\D_R(\H_\m^j(M))\ =\ \Hom_K(\H_\m^j(M),K)\ \cong\ \Hom_K(\H_\n^j(M),K)\ =\ \D_S(\H_\n^j(M)).
		$$
		Since standard graded $K$-algebras are $*$-complete, and $\omega_S\cong S$, the graded version of Grothendieck's local duality \cite[Theorem 3.6.19(c)]{BH} gives
		\begin{equation}\label{eq:iso-graded-D}
		\D_R(\H_\m^j(M))\ \cong\ \D_S(\H_\n^j(M))\ \cong\ \Ext^{n-j}_S(M,S).
		\end{equation}
		
		In the local case, denoting again by $\n$ the maximal ideal of $S$, Grothendieck's local duality \cite[Theorem 11.2.6]{BS} gives $\H_\m^j(M)\cong \D_R(\Ext^{n-j}_S(M,S))$. Applying partial Matlis duality (see \cite[Theorem 10.2.19(ii)]{BS}) we have
		\begin{equation}\label{eq:iso-local-D}
		\D_R(\H_\m^j(M))\ \cong\ \Ext^{n-j}_S(M,S)\otimes_R\widehat{R}. 
		\end{equation}
		
		We claim that
		\begin{equation}\label{eq:dimR-hatR}
		\dim_R\D_R(\H_\m^j(M))\ =\ \dim_R\Ext_S^{n-j}(M,S).
		\end{equation}
	
		In the graded case, this follows from (\ref{eq:iso-graded-D}). In the local case, this follows from (\ref{eq:iso-local-D}) together with \cite[Corollary 2.1.8]{BH}. Having acquired equation (\ref{eq:dimR-hatR}), we obtain that
		\begin{align*}
			S_r\text{-}\!\depth(M)\ &=\ \min\{n-i:\ \dim_R\D_R(\H_{\m}^{n-i}(M))\ge (n-i)-r+1\}\\
			&=\ \min\{n-i:\ \dim_R\Ext^{n-(n-i)}_S(M,S)\ge (n-i)-r+1\}\\
			&=\ n-\max\{i:\ \dim_R\Ext^i_S(M,S)\ge n-i-r+1\},
		\end{align*}
		as desired.
	\end{proof}

    \section{Invariance under completion}\label{sec2}
    
    The goal of this short section is to prove Theorem \ref{ThmC} (that is, Theorem \ref{Thm:completion}). We begin with the following elementary observation.
    \begin{Remark}\label{Rem:Mat}
    	Let $M$ be a finitely generated $\widehat{R}$-module. First, note that $M$ can be regarded as a $R$-module by restriction of scalars via the inclusion map $R\rightarrow\widehat{R}$. Since $\widehat{R}$ is $\m$-adically complete, so is $M$. That is, $M\cong\widehat{\!M\!}$. Since $M$ is also finitely generated as $\widehat{R}$-module, it follows immediately from \cite[Theorem 8.7]{Mat} that $\widehat{\!M\!}\cong M\otimes_R\widehat{R}$, and therefore
    	$$
    	M\ \cong\ M\otimes_R\widehat{R}.
    	$$
    \end{Remark}
    
    For the proof of Theorem \ref{Thm:completion}, we need the following lemma.
    \begin{Lemma}\label{Lem:changeofrings}
    	Let $\varphi:(R,\m)\rightarrow(S,\n)$ be a homomorphism of Noetherian local rings, let $M$ be a $R$-module and let $N$ be a faithfully flat $R$-module which is also a finitely generated $S$-module. Then
    	$$
    	\dim_S(M\otimes_RN)\ =\ \dim_R(M)+\dim_S(N/\m N).
    	$$
    	In particular, if $\varphi$ is faithfully flat, then
    	$$
    	\dim_S(M\otimes_RS)\ =\ \dim_R(M)+\dim(S/\m S).
    	$$
    \end{Lemma}
    \begin{proof}
    	If $M$ is finitely generated, then the statement is just \cite[Theorem A.11(b)]{BH}. Suppose that $M$ is not finitely generated. It is known that $M$ is the direct limit of its finitely generated submodules (see \cite[Example 6.157(iii)]{Rot}). Say $M\cong\lim\limits_{\substack{\longrightarrow\\ \ell}}M_\ell$.
    	
    	Direct limits commute with localizations. Indeed, direct limits commute with tensor products \cite[Theorem 6.159]{Rot}, and the localization of a $R$-module $N$ at a prime $P$ is $N_P=N\otimes_RR_P$. Thus $M_P\cong(\lim\limits_{\substack{\longrightarrow\\ \ell}}M_\ell)_P\cong\lim\limits_{\substack{\longrightarrow\\ \ell}}(M_\ell)_P$. We claim that
    	$$
    	\Supp(M)\ =\ \bigcup_\ell\Supp(M_\ell).
    	$$
    	
    	Let $P\in\Supp(M_h)$. Since $M_h\subset M$, then $M_P\ne0$ and $P\in\Supp(M)$. Conversely, if $P\in\Supp(M)$, then $\lim\limits_{\substack{\longrightarrow\\ \ell}}(M_\ell)_P\ne0$ and so $P\in\Supp(M_h)$ for some $h$.
    	
    	Hence $\dim_R(M)=\sup_\ell\{\dim_R(M_\ell)\}$. Using again \cite[Theorem 6.159]{Rot} and the finitely generated case, we have
    	\begin{align*}
    		\dim_S(M\otimes_RN)\ &=\ \dim_S(\lim\limits_{\substack{\longrightarrow\\ \ell}}(M_\ell\otimes_RN))\ =\ \sup_\ell\{\dim_S(M_\ell\otimes_RN)\}\\
    		&=\ \sup_\ell\{\dim_R(M_\ell)+\dim_S(N/\m N)\}\\
    		&=\ \sup_\ell\{\dim_R(M_\ell)\}+ \dim_S(N/\m N)\\
    		&=\ \dim_R(M)+\dim_S(N/\m N),
    	\end{align*}
    	as desired.
    \end{proof}

    Now, we are in the position to prove Theorem \ref{Thm:completion}.
    \begin{proof}[Proof of Theorem \ref{Thm:completion}]
    	Taking into account the definition of Serre depth, it is enough to show that
    	\begin{equation}\label{eq:dim-enough}
    		\dim_{\widehat{R}}\D_{\widehat{R}}(\H_{\widehat{\m}}^j(\,\widehat{\!M\!}\,))\ =\ \dim_R\D_R(\H_\m^j(M)),\quad\text{for all}\ j\ge0.
    	\end{equation}
        
    	By \cite[Proposition 3.5.4(d)]{BH}, we have the isomorphisms of $R$-modules
    	\begin{equation}\label{eq:iso-H}
    		\H_\m^j(M)\ \cong\ \H_\m^j(M)\otimes_R\widehat{R}\ \cong\ \H^j_{\widehat{\m}}(\,\widehat{\!M\!}\,),\quad\text{for all}\ j\ge0.
    	\end{equation}
    
    	By \cite[Theorem 18.6(iii)]{Mat} (see also \cite[Exercise 3.2.14]{BH}) there is a canonical isomorphism of $R$-modules
    	\begin{equation}\label{eq:iso-E}
    		\E_R(K)\ \cong\ \E_R(K)\otimes_R\widehat{R}\ \cong\ \E_{\widehat{R}}(K).
    	\end{equation}
    	
    	Next, using the isomorphisms (\ref{eq:iso-H}) and the fact that $\widehat{R}$ is a $(R,\widehat{R})$-bimodule, the tensor-hom adjunction  (see \cite[Theorem 6.134]{Rot})  implies that
    	\begin{equation}\label{eq:iso-HomTen}
    		\begin{aligned}
    			\D_{\widehat{R}}(\H_{\widehat{\m}}^j(\,\widehat{\!M\!}\,))\ &=\ \Hom_{\widehat{R}}(\H_{\widehat{\m}}^j(\,\widehat{\!M\!}\,),\E_{\widehat{R}}(K))\\
    			&\cong\ \Hom_{\widehat{R}}(\H_{\m}^j(M)\otimes_R\widehat{R},\E_{\widehat{R}}(K))\\
    			&\cong\ \Hom_R(\H_\m^j(M),\Hom_{\widehat{R}}(\widehat{R},\E_{\widehat{R}}(K))),
    		\end{aligned}
    	\end{equation}
    	for all $j\ge0$.
    	
    	Using that $\Hom_{\widehat{R}}(\widehat{R},\E_{\widehat{R}}(K))\cong\E_{\widehat{R}}(K)$ and since by equation (\ref{eq:iso-H}) we have the isomorphism $\E_R(K)\cong\E_{\widehat{R}}(K)$, the above formula (\ref{eq:iso-HomTen}) implies that
    	\begin{equation}\label{eq:Dhom}
    		\D_{\widehat{R}}(\H_{\widehat{\m}}^j(\,\widehat{\!M\!}\,))\ \cong\ \Hom_R(\H_\m^j(M),\E_{R}(K))\ =\ \D_R(\H_\m^j(M))
    	\end{equation}
    	as $R$-modules, for all $j\ge0$.
    	
    	Since the inclusion map $(R,\m)\rightarrow(\widehat{R},\widehat{\m})$ is a faithfully flat homomorphism of Noetherian local rings (see \cite[Theorem 8.8]{Mat}), Lemma \ref{Lem:changeofrings} yields that
    	\begin{equation}\label{eq:dim-map}
    		\dim_{\widehat{R}}(\D_{R}(\H_{\m}^j(M))\otimes_R\widehat{R})\ =\ \dim_R\D_R(\H_\m^j(M))+\dim(\widehat{R}/\m\widehat{R}).
    	\end{equation}
    	
    	Using that $\widehat{R}/\m\widehat{R}=\widehat{R}/\widehat{\m}\cong K$ is a field, we have $\dim(\widehat{R}/\m\widehat{R})=0$. Thus, to conclude the proof, it is enough to show that
    	$$
    	\D_{R}(\H_{\m}^j(M))\otimes_R\widehat{R}\ \cong\ \D_{\widehat{R}}(\H_{\widehat{\m}}^j(\,\widehat{\!M\!}\,)).
        $$
    	
    	By (\ref{eq:Dhom}), this is equivalent to
    	$$
    	\D_{\widehat{R}}(\H_{\widehat{\m}}^j(\,\widehat{\!M\!}\,))\ \cong\ \D_{\widehat{R}}(\H_{\widehat{\m}}^j(\,\widehat{\!M\!}\,))\otimes_{R}\widehat{R}.
    	$$
    	
    	This isomorphism follows from Remark \ref{Rem:Mat} because $\D_{\widehat{R}}(\H_{\widehat{\m}}^j(\,\widehat{\!M\!}\,))$ is a finitely generated $\widehat{R}$-module. In fact, since $M$ is finitely generated as a $R$-module, then $\H_\m^j(M)$ is Artinian as a $R$-module (see \cite[Proposition 3.5.4(a)]{BH}). Equation (\ref{eq:iso-H}) implies that $\H_{\widehat{\m}}^j(\,\widehat{\!M\!}\,)$ is an Artinian $\widehat{R}$-module as well. Finally, since $\widehat{R}$ is complete, $\D_{\widehat{R}}(\H_{\widehat{\m}}^j(\,\widehat{\!M\!}\,))$ is a finitely generated $\widehat{R}$-module (see \cite[Theorem 3.2.12]{BH}), as desired.
    \end{proof}

    The argument of the proof gives the following additional information.
    \begin{Corollary}
    	Let $(R,\m)$ be a local ring, and $M$ be a finitely generated $R$-module. Then
    	$$\D_R(\H_\m^j(M))\ \cong\ \D_R(\H_\m^j(M))\otimes_R\widehat{R}\ \cong\ \D_{\widehat{R}}(\H_{\widehat{\m}}^j(\,\widehat{\!M\!}\,)).$$
    \end{Corollary}
    
    \section{Serre's conditions}\label{sec3}
    
    In this section, we relate the Serre depth to the famous Serre's conditions. The main result (Theorem \ref{ThmA}) improves several scattered results in the literature, see, for instance: Schenzel (\cite[Lemma 1.9(d)]{Sc2} and \cite[Lemma 3.2.1]{Sc1}), Vasconcelos (\cite[Proposition 3.51]{V}), and Dao, Ma and Varbaro \cite[Proposition 2.11]{DMV}.
    
    As in the previous section, we assume that $(R,\m)$ be either a Noetherian local ring or a standard graded $K$-algebra with unique (homogeneous, if $R$ is a $K$-algebra) maximal ideal $\m$ and residue class field $R/\m\cong K$. Let $r\ge1$ be a positive integer. Recall that a $R$-module $M$ is said to satisfy the $r$-th \textit{Serre's condition} $(S_r)$ if
    $$
    \depth_{R_P}(M_P)\ \ge\ \min\{r,\,\dim_{R_P}(M_P)\},\quad\text{for all}\ P\in\Supp(M).
    $$
    
    For instance, $M$ satisfies $(S_1)$ if and only if $M$ does not have embedded associated primes. Moreover, $M$ satisfies $(S_r)$ for all $r\ge1$ if and only if $M$ is Cohen-Macaulay.\smallskip
    
    Recall that a $R$-module $M$ is called \textit{unmixed} if for all associated prime ideals $P\in\Ass_R(M)$ we have $\dim(R/P)=\dim_R(M)$. Whereas, we say that $M$ is \textit{equidimensional} if $\dim_R(M)=\dim(R/P)$ for all minimal associated primes $P$ of $M$. Clearly, unmixed modules are equidimensional. Note that an equidimensional module satisfying $(S_1)$ is unmixed.\smallskip
    
    Next, we prove a part of Theorem \ref{ThmA}.
    \begin{Theorem}\label{Thm:Serre}
    	Let $(R,\m)$ be a Noetherian local ring or a standard graded $K$-algebra. Assume that $R$ is a homomorphic image of a Gorenstein ring. Let $M$ be a finitely generated $R$-module, which we assume is homogeneous if $R$ is a $K$-algebra. Then, the following conditions are equivalent.
    	\begin{enumerate}
    		\item[\textup{(a)}] $M$ is equidimensional and satisfies Serre's condition $(S_r)$.
    		\item[\textup{(b)}] $S_r\text{-}\!\depth(M)=\dim(M)$.
    	\end{enumerate}
    \end{Theorem}

    This result shows, under mild assumptions, that the $r$-th Serre depth of a finitely generated $R$-module $M$ measures how far $M$ is from being equidimensional and satisfying Serre’s condition $(S_r)$, analogously to how the depth measures the failure of $M$ to be Cohen-Macaulay.\smallskip
     
     For the proof of Theorem \ref{Thm:Serre}, we need several lemmas.\smallskip
    
    \begin{Lemma}\label{Lem:serre1}
    	Let $R$ be a Noetherian catenary ring, and let $M$ be an equidimensional $R$-module. Then
    	$$
    	\dim_R(M)\ =\ \dim_{R_P}(M_P)+\dim(R/P),\quad\text{for all}\ P\in\Supp(M).
    	$$
    \end{Lemma}
    \begin{proof}
    	Fix a prime $P\in\Supp(M)$. Let $Q\subset P$ be a minimal associated prime of $M$. Then $\dim_R(M)=\dim(R/Q)$ and $QR_P$ is a minimal associated prime ideal of $M_P$. Hence $\dim_{R_P}(M_P)=\dim_{R_P}(R_P/QR_P)$. Since $R$ is a catenary ring by assumption, we have $\dim(R/Q)=\dim(R_P/QR_P)+\dim(R/P)$. Hence,
    	$$
    	\dim_R(M)=\dim(R/Q)=\dim(R_P/QR_P)+\dim(R/P)=\dim_{R_P}(M_P)+\dim(R/P),
    	$$
    	as desired.
    \end{proof}

    \begin{Lemma}\label{Lem:serre2}
    	Let $(R,\m)$ be a Noetherian local ring or a standard graded $K$-algebra. Assume that $R$ is a homomorphic image of a Gorenstein ring $S$. Let $M$ be a finitely generated $R$-module, which we assume is homogeneous if $R$ is a $K$-algebra. Then, $$\Supp_R(\D_R(\H_\m^j(M)))\ =\ \Supp_R(\Ext^{\dim(S)-j}_S(M,S))$$ and
    	$$
    	P\in\Supp_R(\D_R(\H_\m^j(M)))\ \ \Longleftrightarrow\ \  \D_{R_P}(\H_{PR_P}^{j-\dim(R/P)}(M_P))\ne0,
    	$$
    	for all $P\in\Supp_R(M)$.
    \end{Lemma}
    \begin{proof}
    	In the graded setting, the first statement is trivial in view of equation (\ref{eq:iso-graded-D}). When $R$ is local, equation (\ref{eq:iso-local-D}) holds, and the first statement holds due to the following more general fact. Let $N,U$ be $R$-modules such that $N\cong U\otimes_R\widehat{R}$. Then $\Supp_R(N)=\Supp_R(U)$. Indeed, let $P\in\Spec(R)$. By \cite[Proposition 10.32]{Rot} we have
    	$$
    	N_P\ \cong\ (U\otimes_R\widehat{R})_P\ \cong\ U_P\otimes_{R_P}\widehat{R}_P.
    	$$
    	Since $\widehat{R}$ is a faithfully flat $R$-module, then $\widehat{R}_P$ is also a faithfully flat $R_P$-module. Hence $N_P\ne0$ if and only if $U_P\ne0$. This yields that $\Supp_R(N)=\Supp_R(U)$, as desired.
    	 
    	Now for the second claim, let $P\in\Supp_R(\Ext^{n-j}_S(M,S))$ be a prime. Since $M$ is finitely generated as a $R$-module, and $R\cong S/I$ for some $I\subset S$, $M$ is also finitely generated as a $S$-module. Thus, \cite[Proposition 10.42]{Rot} gives
    	\begin{equation}\label{eq:iso2}
    		(\Ext_S^{\dim(S)-j}(M,S))_P\ \cong\ \Ext_{S_P}^{\dim(S)-j}(M_P,S_P).
    	\end{equation}
    	Let $\m_P=PR_P$. The ring $(R_P,\m_P)$ is again the quotient of a Gorenstein ring. In fact $R_P\cong(S/I)_P=S_P/I_P$ and the localization of a Gorenstein ring is again Gorenstein (see \cite[Proposition 3.1.19(a)]{BH}). Gorenstein rings are Cohen-Macaulay, and thus universally catenary (see \cite[Theorem 2.12]{BH}). Furthermore, it is easily verified (see the comments before \cite[Theorem 2.12]{BH}) that a ring $T$ is universally catenary if and only if any finitely generated $T$-algebra is such. In particular, $R$ is universally catenary being a quotient of $S$. Hence $S$ and $R$ are catenary. Since $R\cong S/I$ then $P\cong Q/I$ with $Q$ a prime of $S$ containing $I$. So $R/P\cong(S/I)/(Q/I)\cong S/Q$, $S_P=S_Q$ and Lemma \ref{Lem:serre1} yields that $\dim(S_P)=\dim(S)-\dim(S/Q)=\dim(S)-\dim(R/P)$. In the local case equation (\ref{eq:iso-local-D}) yields that
    	\begin{equation}\label{eq:iso3}
    		\begin{aligned}
    			\D_{R_P}(\H_{\m_P}^{j-\dim(R/P)}(M_P))\ &\cong\ \Ext^{\dim(S_P)-(j-\dim(R/P))}_{S_P}(M_P,S_P)\otimes_{R_P}\widehat{R_P}\\
    			&=\ \Ext_{S_P}^{\dim(S)-j}(M_P,S_P)\otimes_{R_P}\widehat{R_P}.
    		\end{aligned}
    	\end{equation}
        Since $\widehat{R_P}$ is a faithfully flat $R_P$-module, (\ref{eq:iso3}) implies that
        $$
        \D_{R_P}(\H_{\m_P}^{j-\dim(R/P)}(M_P))\ne0\ \ \Longleftrightarrow\ \ \Ext_{S_P}^{\dim(S)-j}(M_P,S_P)\ne0.
        $$
        By equation (\ref{eq:iso2}), this is equivalent to $(\Ext_S^{\dim(S)-j}(M,S))_P\ne0$. Combining this with the first statement, this is further equivalent to $P\in\Supp_R(\D_R(\H_\m^j(M)))$, as wanted.\smallskip
        
        In the graded case, we can proceed similarly, by using equation (\ref{eq:iso-graded-D}).
    \end{proof}

    \begin{Lemma}\label{Lem:serre3}
    	Let $(R,\m)$ be a Noetherian local ring or a standard graded $K$-algebra. Assume that $R$ is a homomorphic image of a Gorenstein ring. Let $M$ be a finitely generated $R$-module, which we assume is homogeneous if $R$ is a $K$-algebra. If $S_1\text{-}\!\depth(M)=\dim(M)$, then $M$ is unmixed, in particular equidimensional.
    \end{Lemma}
    \begin{proof}
    	Suppose by contradiction that $M$ is not unmixed. Then there exists a prime $P\in\Ass_R(M)$ with $\dim(R/P)<\dim_R(M)$. Localizing at $P$, we obtain that $\depth_{R_P}(M_P)=0$ because $\m_P=PR_P\in\Ass_{R_P}(M_P)$. By \cite[Theorem 3.5.7(b)]{BH} we have $\H_{\m_P}^0(M_P)\ne0$. Hence, Lemma \ref{Lem:serre2} implies that $P$ belongs to the support of $\D_R(\H_\m^{\dim(R/P)}(M))$ and this shows that
    	$$
    	\dim_R\D_R(\H_\m^{\dim(R/P)}(M))\ \ge\ \dim(R/P).
    	$$
    	
    	Therefore $S_1\text{-}\!\depth(M)\le\dim(R/P)<\dim_R(M)$, but this contradicts the assumption. Hence $M$ must be unmixed.
    \end{proof}
    
    We are now in the position to prove Theorem \ref{Thm:Serre}.
    \begin{proof}[Proof of Theorem \ref{Thm:Serre}]
    	Let $n=\dim(R)$.\smallskip
    	
    	(a) $\Rightarrow$ (b) Assume that $M$ is equidimensional and satisfies Serre's condition $(S_r)$. Suppose  that
    	$$
    	S_r\text{-}\!\depth_R(M)=j<\dim_R(M).
    	$$
    	
    	Then $\dim_R\D_R(\H_\m^j(M))\ge j-r+1$ and so $\D_R(\H_\m^j(M))\ne0$. We choose a prime $P\in\Supp_R(\D_R(\H_\m^j(M)))$ such that $\dim(R/P)=\dim_R\D_R(\H_\m^j(M))$. Lemma \ref{Lem:serre2} implies that $\D_{R_P}(\H_{\m_P}^{j-\dim(R/P)}(M_P))\ne0$. Thus $\H_{\m_P}^{j-\dim(R/P)}(M_P)\ne0$ as well, and hence \cite[Theorem 3.5.7]{BH} gives that $\depth_{R_P}(M_P)\le j-\dim(R/P)$. Since $M$ is equidimensional, $R$ is catenary and $j<\dim_R(M)$, Lemma \ref{Lem:serre1} implies that
    	$$
    	\depth_{R_P}(M_P)\ \le\  j-\dim(R/P)\ <\ \dim_R(M)-\dim(R/P)=\ \dim_{R_P}(M_P). 
    	$$
    	Moreover,
    	$$
    	\depth_{R_P}(M_P)\ \le\ j-\dim(R/P)\ =\ j-\dim_R\D_R(\H_\m^j(M))\ \le\ r-1\ <\ r.
    	$$
    	Altogether, $\depth_{R_P}(M_P)<\min\{r,\,\dim_{R_P}(M_P)\}$. This contradicts the fact that $M$ satisfies Serre's condition $(S_r)$, as desired.\smallskip
    	
    	(b) $\Rightarrow$ (a) Assume that $S_r\text{-}\!\depth_R(M)=\dim_R(M)$. By Proposition \ref{Prop:inequalities}(a) we have $S_1\text{-}\!\depth_R(M)=\dim_R(M)$, too. Then, Lemma \ref{Lem:serre3} implies that $M$ is equidimensional. Suppose for a contradiction that $$\depth_{R_P}(M_P)\ <\ \min\{r,\,\dim_{R_P}(M_P)\}$$ for some prime $P\in\Supp_R(M)$. Let $\m_P=PR_P$ and $j=\dim(R/P)+\depth_{R_P}(M_P)$. By \cite[Theorem 3.5.7(b)]{BH}, $\H_{\m_P}^{\depth_{R_P}(M_P)}(M_P)\ne0$. Hence $\D_R(\H_{\m_P}^{\depth_{R_P}(M_P)}(M_P))\ne0$ too and Lemma \ref{Lem:serre2} implies that $P\in\Supp_R(\D_R(\H_\m^j(M)))$. Since $\depth_{R_P}(M_P)<r$, we obtain that
    	\begin{equation}\label{eq:ineqP}
    		\dim(R/P)\ =\ j-\depth_{R_P}(M_P)\ >\ j-r.
    	\end{equation}
    	Hence $\dim_R\D_R(\H_\m^j(M))\ge\dim(R/P)\ge j-r+1$. By the definition of $r$-th Serre depth, we get $j\ge S_r\text{-}\!\depth_R(M)=\dim_R(M)$. Since $\depth_{R_P}(M_P)<\dim_{R_P}(M_P)$, $M$ is equidimensional and $R$ is catenary, then Lemma \ref{Lem:serre1} and equation (\ref{eq:ineqP}) imply that
    	\begin{align*}
    		\dim(R/P)\ &=\  j-\depth_{R_P}(M_P)\\&\ge\ S_r\text{-}\!\depth_R(M)-\depth_{R_P}(M_P)\\
    		&>\ \dim_R(M)-\dim_{R_P}(M_P)\\&=\ \dim(R/P),
    	\end{align*}
    	which is a contradiction. The assertion follows.
    \end{proof}
    
    Lemma \ref{Lem:serre2} gives the following additional information.
    \begin{Corollary}
    	Let $(R,\m)$ be a Noetherian local ring or a standard graded $K$-algebra. Assume that $R$ is a homomorphic image of a Gorenstein ring. Let $M$ be a finitely generated $R$-module, which we assume is homogeneous if $R$ is a $K$-algebra. Then
    	$$
    	S_r\text{-}\!\depth_{R_P}(M_P)\ \ge\ S_{r+\dim(R/P)}\text{-}\!\depth_R(M)-\dim(R/P),
    	$$
    	for all primes $P\in\Supp_R(M)$ and $r\ge1$. Equality holds for a given $P\in\Supp_R(M)$ if $S_{r-\dim(R/P)}\text{-}\!\depth_R(M)=\dim_R(M)$.
    \end{Corollary}
    \begin{proof}
    	Fix $P\in\Supp_R(M)$, and put $s=S_r\text{-}\!\depth_{R_P}(M_P)$. Let $j=s+\dim(R/P)$. Since $\D_{R_P}(\H_{PR_P}^{j-\dim(R/P)}(M_P))\ne0$, Lemma \ref{Lem:serre2} gives that $P\in\Supp_R(\D_R(\H_\m^j(M)))$.
    	
    	By definition of $r$-th Serre depth, we have
    	$$
    	d\ =\ \dim_{R_P}\D_{R_P}(\H_{PR_P}^{s}(M_P))\ \ge\ s-r+1.
    	$$
    	Let $Q\in\Supp_{R_P}(\D_{R_P}(\H_{PR_P}^{s}(M_P)))$ with $\dim(R_P/Q)=d$. Let $\pi:R\rightarrow R_P$ be the localization map, and let $Q'=\pi^{-1}(Q)\in\Spec(R)$. Then, Lemma \ref{Lem:serre2} gives that
    	\begin{equation}\label{eq:<=>contorta}
    		Q'\in\Supp_R(\D_R(\H_\m^{s+\dim(R/P)}(M)))\Longleftrightarrow \D_{R_{Q'}}(\H_{Q'R_{Q'}}^{s+\dim(R/P)-\dim(R/Q')}(M_{Q'}))\ne0.
    	\end{equation}
    	Since $R_{Q'}=(R_P)_Q$, $Q'R_{Q'}=Q(R_P)_Q$ and $M_{Q'}=(M_P)_Q$, Lemma \ref{Lem:serre2} also implies that
    	$$
    	Q\in\Supp_{R_P}(\D_{R_P}(\H_{PR_P}^{s}(M)))\ \Longleftrightarrow\ \D_{R_{Q'}}(\H_{Q'R_{Q'}}^{s-\dim(R_P/Q)}(M_{Q'}))\ne0.
    	$$
    	But we choose $Q\in\Supp_{R_P}(\D_{R_P}(\H_{PR_P}^{s}(M)))$. Hence $\D_{R_{Q'}}(\H_{Q'R_{Q'}}^{s-\dim(R_P/Q)}(M_{Q'}))\ne0$. Note that $\dim(R/Q')=\dim(R_P/Q)+\dim(R/P)$. So
    	$$
    	\D_{R_{Q'}}(\H_{Q'R_{Q'}}^{s+\dim(R/P)-\dim(R/Q')}(M_{Q'}))\ =\ \D_{R_{Q'}}(\H_{Q'R_{Q'}}^{s-\dim(R_P/Q)}(M_{Q'}))\ \ne\ 0.
    	$$
    	Hence (\ref{eq:<=>contorta}) implies that $Q'\in\Supp_R(\D_R(\H_\m^{s+\dim(R/P)}(M)))$. Thus
    	\begin{align*}
    		\dim_R\D_R(\H_\m^{s+\dim(R/P)}(M))\ &\ge\ \dim(R/Q')\\
    		&\ge\ d\ \ge\ s-r+1\\
    		&=\ (s+\dim(R/P))-(r+\dim(R/P))+1.
    	\end{align*}
    	Consequently,
    	\begin{equation}\label{eq:ineq-r+dim(R/P)}
    		S_{r+\dim(R/P)}\text{-}\!\depth_R(M)\ \le\ s+\dim(R/P)\ =\ S_r\text{-}\!\depth_{R_P}(M_P)+\dim(R/P).
    	\end{equation}

    	Finally, assume in addition that $S_{r+\dim(R/P)}\text{-}\!\depth_R(M)=\dim_R(M)$ for a fixed $P\in\Supp_R(M)$. By Theorem \ref{Thm:Serre}, $M$ satisfies Serre's condition $(S_{r+\dim(R/P)})$ and so it satisfies $(S_1)$. Hence $M$ is equidimensional. Since $R$ is catenary, Lemma \ref{Lem:serre1} together with inequality (\ref{eq:ineq-r+dim(R/P)}) yield that
    	$$
    	S_r\text{-}\!\depth_{R_P}(M_P)\ \ge\ \dim_R(M)-\dim(R/P)\ =	\ \dim_{R_P}(M_P).
    	$$
    	Using Proposition \ref{Prop:inequalities}(a), the desired equality follows.
    \end{proof}
	
	One may ask if the conditions (a) and (b) stated in Theorem \ref{Thm:Serre} are actually equivalent for any Noetherian local ring or standard graded $K$-algebra. Next, we prove that the implication (b) $\Rightarrow$ (a) holds in general. This result completes the proof of Theorem \ref{ThmA}.
	
	\begin{Theorem}\label{Thm:general-implication}
		Let $(R,\m)$ be a Noetherian local ring or a standard graded $K$-algebra. Let $M$ be a finitely generated $R$-module, which we assume is homogeneous if $R$ is a $K$-algebra. Assume that $S_r\text{-}\!\depth(M)=\dim(M)$. Then $M$ is unmixed, hence equidimensional, and satisfies Serre's condition $(S_r)$.
	\end{Theorem}
    \begin{proof}
    	If $R$ is a $K$-algebra, Theorem \ref{Thm:Serre} applies and there is nothing to prove. We therefore assume that $R$ is local. Now, Theorem \ref{Thm:completion}, \cite[Corollary 2.1.8(a)]{BH} and the assumption imply that
    	$$
    	S_r\text{-}\!\depth_{\widehat{R}}(\,\widehat{\!M\!}\,)\ =\ S_r\text{-}\!\depth_R(M)\ =\ \dim_R(M)\ =\ \dim_{\widehat{R}}(\,\widehat{\!M\!}\,).
    	$$
    	
    	By Cohen's structure theorem (\cite[Theorem 29.4(ii)]{Mat}), since regular rings are Gorenstein (\cite[Proposition 3.1.20]{BH}), $\widehat{R}$ is a homomorphic image of a Gorenstein ring. Hence, Theorem \ref{Thm:Serre} and the above equation imply that $\widehat{\!M\!}$ is equidimensional and satisfies Serre's condition $(S_r)$. Hence $\widehat{\!M\!}$ is unmixed, because it is equidimensional and satisfy $(S_1)$.
    	
    	Let us prove that $M$ is unmixed. Since $M$ is finitely generated, \cite[Theorem 8.7]{Mat} yields that $\widehat{\!M\!}\cong M\otimes_R\widehat{R}$. Next, recalling that the inclusion map $R\rightarrow\widehat{R}$ is faithfully flat, \cite[Theorem 23.2(ii)]{Mat} gives
    	$$
    	\Ass_{\widehat{R}}(\,\widehat{\!M\!}\,)\ =\ \bigcup_{P\in\Ass_R(M)}\Ass_{\widehat{R}}(\widehat{R}/P\widehat{R}).
    	$$
    	Let $P\in\Ass_R(M)$. Note that $R/P\otimes_{R}\widehat{R}\cong\widehat{R}/P\widehat{R}$. Since $\widehat{\m}=\m\widehat{R}$, Lemma \ref{Lem:changeofrings} now gives
    	$$
    	\dim(\widehat{R}/P\widehat{R})\ =\ \dim_{\widehat{R}}(R/P\otimes_{R}\widehat{R})\ =\ \dim_R(R/P)+\dim(\widehat{R}/\m\widehat{R})\ =\ \dim(R/P).
    	$$
    	Since $\widehat{\!M\!}$ is unmixed we obtain that $\dim(R/P)=\dim(\widehat{R}/P\widehat{R})=\dim_{\widehat{R}}(\,\widehat{\!M\!}\,)$ for all $P\in\Ass_R(M)$. We conclude that $M$ is indeed unmixed.\smallskip
    	
    	It remains to prove that $M$ satisfies $(S_r)$. To this end, let $P\in\Supp_R(M)$. Choose a minimal prime ideal $Q\in\Spec(\widehat{R})$ of $P\widehat{R}$. Then $\dim(\widehat{R}_Q/P\widehat{R}_Q)=0$. Recall that the inclusion map $\varphi:R\rightarrow\widehat{R}$ is faithfully flat. Since $\varphi^{-1}(Q)=P$ by construction, it follows that the map $\varphi_Q:R_P\rightarrow\widehat{R}_Q$ is again faithfully flat. Note that
    	$$
    	\widehat{\!M\!}_P\ \cong\ (M\otimes_R\widehat{R})_P\ \cong\ M_P\otimes_{R_P}\widehat{R}_Q.
    	$$
    	 So, Lemma \ref{Lem:changeofrings} yields that
    	\begin{align*}
    		\dim_{\widehat{R}_Q}(\,\widehat{\!M\!}_P)\ &=\  \dim_{\widehat{R}_Q}(M_P\otimes_{R_P}\widehat{R}_Q)\ =\  \dim_{R_P}(M_P)+\dim(\widehat{R}_Q/P\widehat{R}_Q)\\
    		&=\ \dim_{R_P}(M_P).
    	\end{align*}
    
    	Similarly, \cite[Proposition 1.2.16(a)]{BH} gives $\depth_{\widehat{R}_Q}(\,\widehat{\!M\!}_Q)=\depth_{R_P}(M_P)$. Using that $\widehat{\!M\!}$ satisfies Serre's condition $(S_r)$, we obtain that
    	$$
    		\depth_{R_P}(M_P)\ =\ \depth_{\widehat{R}_Q}(\,\widehat{\!M\!}_Q)\ \ge\ \min\{r,\,\dim_{\widehat{R}_Q}(\,\widehat{\!M\!}_Q)\}\ =\ \min\{r,\,\dim_{R_P}(M_P)\}.
    	$$
    	Since $P\in\Supp_R(M)$ is arbitrary, it follows that $M$ satisfies Serre's condition $(S_r)$, as desired.
    \end{proof}
	
	However, the converse statement is false in general, as we show next. The idea is to find an example that has the same properties as Ogoma’s famous example of a normal Noetherian local domain that is not catenary. This was achieved in \cite{HRW}.
	\begin{Example}\label{Ex:HRW}
		\rm It is proved in \cite[Theorem 19.11]{HRW} by Heinzer, Rotthaus and Wiegand that there exists a $2$-dimensional Noetherian local domain $(R,\m)$ (hence $R$ is unmixed and satisfies $(S_1)$), such that $\widehat{R}$ does not satisfy Serre's condition $(S_1)$.
		
		By Cohen structure theorem, $\widehat{R}$ is a homomorphic image of a Gorenstein ring. Since $\widehat{R}$ does not satisfy $(S_1)$, Theorem \ref{Thm:Serre} implies that $S_1\text{-}\!\depth_{\widehat{R}}(\,\widehat{\!R\!}\,)<\dim(\,\widehat{\!R\!}\,)$. Theorem \ref{Thm:completion} and \cite[Corollary 2.1.8(a)]{BH} then imply that
		$$
		S_1\text{-}\!\depth_{R}(R)\ =\ S_1\text{-}\!\depth_{\widehat{R}}(\,\widehat{\!R\!}\,)\ <\ \dim(\,\widehat{\!R\!}\,)\ =\ \dim(R),
		$$
		but $R$ satisfies $(S_1)$. Hence, the converse of Theorem \ref{Thm:general-implication} is not true in general.
	\end{Example}
	
	If we consider only quotient rings instead of general modules, Theorem \ref{Thm:Serre} takes the following nicer form.
	\begin{Corollary}\label{Cor:Serre}
		Let $(R,\m)$ be a Noetherian local ring or a standard graded $K$-algebra. Assume that $R$ is a homomorphic image of a Gorenstein ring. Let $I\subset R$ be an ideal, which we assume is homogeneous if $R$ is a $K$-algebra. Let $r\ge2$. Then, the following conditions are equivalent.
		\begin{enumerate}
			\item[\textup{(a)}] $R/I$ satisfies Serre's condition $(S_r)$.
			\item[\textup{(b)}] $S_r\text{-}\!\depth(R/I)=\dim(R/I)$.
		\end{enumerate}
	\end{Corollary}
    \begin{proof}
    	Let $r\ge2$. By Theorem \ref{Thm:Serre}, the following conditions are equivalent.
    	\begin{enumerate}
    		\item[\textup{(a)$'$}] $R/I$ is equidimensional and satisfies Serre's condition $(S_r)$.
    		\item[\textup{(b)}] $S_r\text{-}\!\depth(R/I)=\dim(R/I)$.
    	\end{enumerate}
        So, it remains to prove that the conditions (a) and (a)$'$ are in fact equivalent. It is clear that (a)$'$ $\Rightarrow$ (a). Conversely, assume that (a) holds. As noted before, $R$ is universally catenary, because it is the quotient of a Gorenstein ring. Since quotients of universally catenary rings are universally catenary, hence catenary, we have that $R/I$ is catenary. Since $r\ge2$, $R/I$ in particular satisfies Serre's condition $(S_2)$. It follows from \cite[Corollary 5.10.9]{EGAIV65} that $R/I$ is equidimensional.
    \end{proof}
	
	Unfortunately, for arbitrary finitely generated modules we can not remove the equidimensional condition in statement (a) of Theorem \ref{Thm:Serre}.
	\begin{Example}
		\rm Let $R=K[[u,v,w,x,y]]$ be a power series ring over a field $K$. Then $R$ is a complete ring. The finitely generated $R$-module $M=R/(u,v)\oplus R/(w,x,y)\cong K[[w,x,y]]\oplus K[[u,v]]$ satisfies $(S_2)$, but $S_2\text{-}\!\depth(M)=2<3=\dim(M)$. Note that $M$ is not equidimensional. 
	\end{Example}

    It was proved by Marley and Vassilev in \cite[Lemma 2.8]{MV} that if $R$ is a catenary local ring and $M$ is a finitely generated indecomposable $R$-module satisfying $(S_2)$, then $M$ is equidimensional. Hence, analogously to Corollary \ref{Cor:Serre} we have
    \begin{Corollary}\label{Cor:Serre1}
    	Let $(R,\m)$ be a Noetherian local ring or a standard graded $K$-algebra. Assume that $R$ is a homomorphic image of a Gorenstein ring. Let $M$ be an indecomposable finitely generated $R$-module, which we assume is homogeneous if $R$ is a $K$-algebra. Let $r\ge2$. Then, the following conditions are equivalent.
    	\begin{enumerate}
    		\item[\textup{(a)}] $M$ satisfies Serre's condition $(S_r)$.
    		\item[\textup{(b)}] $S_r\text{-}\!\depth(M)=\dim(M)$.
    	\end{enumerate}
    \end{Corollary}

    We point out that the Corollaries \ref{Cor:Serre} and \ref{Cor:Serre1} are independent statements. In fact, a quotient ring can be a decomposable module over its ambient ring.
	
    \section{Gr\"obner degenerations}\label{sec4}
    
    In this section, we consider the more familiar situation of the standard graded polynomial ring $S=K[x_1,\dots,x_n]$ over a field $K$. First, we establish Theorem \ref{ThmD}.
    
    \begin{Theorem}\label{Thm:IneqMonOrd}
    	Let $S=K[x_1,\dots,x_n]$ be the standard graded polynomial ring over a field $K$, let $I\subset S$ be a homogeneous ideal, and let $<$ be a monomial order on $S$. Then
    	$$
    	S_r\text{-}\!\depth(S/\ini_<(I))\ \le\ S_r\text{-}\!\depth(S/I),
    	$$
    	for all $r\ge1$. Furthermore, equality holds if $\ini_<(I)$ is squarefree.
    \end{Theorem}
    \begin{proof}
    	It is well-known (see, for instance, the paper of Sbarra \cite[Theorem 2.4]{Sb}) that we have the following inequality of dimensions os $K$-vector spaces:
    	\begin{equation}\label{eq:dimSbarra1}
    	\dim_K\H_\m^j(S/I)_i\ \le\ \dim_K\H_\m^j(S/\ini_<(I))_i,\quad\textup{for all}\ i\in\ZZ,\ j\ge0.
    	\end{equation}
    
    	For a $K$-vector space $V$, the graded Matlis $\D_S(V)=\Hom_K(V,K)$ preserves the $K$-vector space dimension. Hence, formulas (\ref{eq:gradedPieceMatlis}) and (\ref{eq:dimSbarra1}) give that
    	\begin{equation}\label{eq:dimSbarra}
    	\dim_K\D_S(\H_\m^j(S/I))_i\ \le\ \dim_K\D_S(\H_\m^j(S/\ini_<(I)))_i,\quad\textup{for all}\ i\in\ZZ,\ j\ge0.
    	\end{equation}
    	
    	The Hilbert function of a finitely generated graded $S$-module $M=\bigoplus_{i\in\ZZ}M_i$ is the function $i\in\ZZ\mapsto\Hilb(M,i)=\dim_K(M_i)$. By \cite[Theorem 4.1.3]{BH}, $\Hilb(M,i)$ eventually agrees with a polynomial $P(x)\in\QQ[x]$ of degree $\dim_S(M)-1$: that is, $\Hilb(M,i)=P(i)$ for all $i\gg0$. Hence, formula (\ref{eq:dimSbarra}) shows that
    	$$
    	\dim_S\D_S(\H_\m^j(S/I))\ \le\ \dim_S\D_S(\H_\m^j(S/\ini_<(I))),\quad\textup{for all}\ j\ge0.
    	$$
    	This yields at once that $S_r\text{-}\!\depth(S/\ini_<(I))\le S_r\text{-}\!\depth(S/I)$, for all $r\ge1$.
    	
    	Finally, it was shown by Conca and Varbaro (see \cite[Theorem 1.3]{CV}) that if in addition $\ini_<(I)$ is a squarefree monomial ideal, then (\ref{eq:dimSbarra1}) is an equality for all $i\in\ZZ$ and $j\ge0$, and in this case by the same reasoning as above we obtain the equality  $S_r\text{-}\!\depth(S/\ini_<(I))=S_r\text{-}\!\depth(S/I)$, for all $r\ge1$.
    \end{proof}
    
    As a consequence, we obtain a new proof of the following well-known fact.
    
    \begin{Corollary}
    	Let $S=K[x_1,\dots,x_n]$ and $I\subset S$ be a homogeneous ideal. Suppose that $\ini_{<}(I)$ satisfies Serre's condition $(S_r)$ for some monomial order $<$. Then $I$ satisfies Serre's condition $(S_r)$, too. The converse holds if $\ini_<(I)$ is squarefree for some monomial order $<$.
    \end{Corollary}
    \begin{proof}
    	By assumption and Theorem \ref{Thm:Serre}, $S_r\text{-}\!\depth(S/\ini_<(I))=\dim(S/\ini_<(I))$. It is well-known that $\dim(S/I)=\dim(S/\ini_<(I))$, see \cite[Theorem 3.3.4(a)]{HHBook}. Combining this fact with Theorem \ref{Thm:IneqMonOrd} and Proposition \ref{Prop:inequalities}(a) we obtain that
    	$$
    	\dim(S/I)=S_r\text{-}\!\depth(S/\ini_<(I))\le S_r\text{-}\!\depth(S/I)\le\dim(S/I).
    	$$
    	Hence $S_r\text{-}\!\depth(S/I)=\dim(S/I)$, which by Theorem \ref{Thm:Serre} indeed means that $S/I$ satisfies Serre's condition $(S_r)$.
    	
    	Conversely, if $\ini_<(I)$ is squarefree the assertion holds because by Theorem \ref{Thm:IneqMonOrd} the Serre depths of $S/I$ and $S/\ini_<(I)$ coincide.
    \end{proof}

    Theorem \ref{Thm:IneqMonOrd} shows that in order to estimate the Serre depth of a homogeneous ideal of $S$, the Serre depth of any of its initial ideals serves as a lower bound. This leads us to the problem of computing the Serre depths of monomial ideals.
    
    Muta and Terai gave a nice formula (\cite[Theorem 3.12]{MT}) for the Serre depth of a squarefree monomial ideal $I\subset S$, in terms of the skeletons of the simplicial complex associated to $I$. This formula in turn generalizes a well-known formula for the depth of $S/I$, with $I\subset S$ squarefree, due to Smith in the pure case (\cite[Theorem 3.7]{Smith}) and independently to Hibi in general (\cite[Corollary 2.6]{Hibi}).
    
    Such a formula for the depth was generalized to any monomial ideal, not necessarily squarefree, by Herzog, Soleyman Jahan and Zheng (\cite[Corollary 1.5]{HSZ}).  Our next goal is to establish a similar formula for the Serre depths of any monomial ideal. To this end, we need to recall a few concepts.
    
    For an integer $n\ge1$, we put $[n]=\{1,\dots,n\}$. Given a monomial $u\in S$, the integer
    $$
    \deg_{x_i}(u)\ =\ \max\{j:\ x_i^j\ \textup{divides}\ u\}
    $$
    is called the \textit{$x_i$-degree} of $u$. For a monomial ideal $I\subset S$, let $\mathcal{G}(I)$ be the unique minimal monomial generating set of $I$. The \textit{bounding multidegree} of $I\subset S$ is the vector
    $$
    {\bf deg}(I)\ =\ (\deg_{x_1}(I),\dots,\deg_{x_n}(I))
    $$
    defined by
    $$
    \deg_{x_i}(I)\ =\ \max_{u\in\mathcal{G}(I)}\deg_{x_i}(u),\quad\text{for all}\ i\in[n].
    $$
    
    Given ${\bf a}=(a_1,\dots,a_n)\in\ZZ^n$, we put ${\bf a}(i)=a_i$ for all $i\in[n]$. Given ${\bf a},{\bf b}\in\ZZ^n$, we write ${\bf a}\ge{\bf b}$ (respectively, ${\bf a}\le{\bf b}$) if ${\bf a}(i)\ge{\bf b}(i)$ (respectively, ${\bf a}(i)\le{\bf b}(i)$) for all $i\in[n]$. Given ${\bf a}\in\NN^n$, we put ${\bf x^a}=\prod_{i\in[n]}x_i^{{\bf a}(i)}$. Here, by $\NN$ we denote as usual the set $\{0,1,2,\dots\}$ of the non-negative integers.
    
    Now, let $I\subset S$ be a monomial ideal, and let ${\bf g}\in\NN^n$ such that ${\bf g}\ge{\bf deg}(I)$. Let $\rho_{\bf g}:\NN^n\rightarrow\NN$ be the function defined by setting
    $$
    \rho_{\bf g}({\bf a})\ =\ |\{s\in[n]:\ {\bf g}(s)={\bf a}(s)\}|,\quad\textup{for all}\ {\bf a}\in\NN^n.
    $$
    
    Let $d=\dim(S/I)$. Following \cite{HSZ}, associated to $I$ and ${\bf g}$ we have the filtration
    $$
    \Sigma_{d}^{\bf g}(I)\subset\Sigma_{d-1}^{\bf g}(I)\subset\cdots\subset \Sigma_{0}^{\bf g}(I)\subset S,
    $$
    where the monomial ideal
    $$
    \Sigma_{i}^{\bf g}(I)\ =\ I+({\bf x^a}:\ {\bf a}\in\NN^n,\ \rho_{\bf g}({\bf a})>i)
    $$
    is called the $i$-th \textit{skeleton} of $I$ with respect to ${\bf g}$. Note that $\Sigma_d^{\bf g}(I)=I$, and moreover $\dim(S/\Sigma_i^{\bf g}(I))=i$, for all $i=0,\dots,d$.
    
    When $I$ is squarefree, writing $I=I_\Delta$ as the Stanley-Reisner ideal of a unique simplicial complex $\Delta$ on $[n]$, and letting ${\bf g}={\bf 1}=(1,1,\dots,1)$, then $\Sigma_{i}^{\bf 1}(I)=I_{\Delta^{(i)}}$ where $\Delta^{(i)}=\{F\in\Delta:|F|\le i\}$ is the $i$-th skeleton of $\Delta$.
    
    As a vast generalization of \cite[Theorem 3.12]{MT}, and the results \cite[Theorem 3.7]{Smith}, \cite[Corollary 2.6]{Hibi} and \cite[Corollary 1.5]{HSZ} mentioned before, we prove Theorem \ref{ThmE}.
    
    \begin{Theorem}\label{Thm:Skeletons}
    	Let $I\subset S$ be a monomial ideal. Then
    	$$
    	S_r\text{-}\!\depth(S/I)\ =\ \max\{i:\ S/\Sigma_i^{\bf g}(I)\ \textit{satisfies Serre's condition}\ (S_r)\},
    	$$
    	for all $r\ge1$ and all ${\bf g}\in\NN^n$ with ${\bf g}\ge{\bf deg}(I)$. Furthermore, $S/\Sigma_{j}^{\bf g}(I)$ satisfies Serre's condition $(S_r)$ for all $j\le S_r\text{-}\!\depth(S/I)$.
    \end{Theorem}
    \begin{proof}
    	Let $d=\dim S/I$ and $t=S_r\text{-}\!\depth(S/I)$. We proceed by induction on $d\ge0$. For $d=0$, the statements hold vacuously. Let $d>0$. To simplify the notation, we put $\Sigma_j=\Sigma_j^{\bf g}(I)$ for all $j=0,\dots,d$. We claim that the following statements hold.
    	\begin{enumerate}
    		\item[(i)]  If $t<d$, then $S_r\text{-}\!\depth(S/\Sigma_{d-1})=t$.
    		\item[(ii)] If $S/I$ satisfies Serre's condition $(S_r)$, then $S/\Sigma_{d-1}$ satisfies  $(S_r)$, too.
    	\end{enumerate}
    	
    	\textit{Proof of} (i): Applying Proposition \ref{Prop:inequalities}(\ref{eq:Sr-ses2})-(b) to the short exact sequence
    	$$
    	0\rightarrow\Sigma_{d-1}/I\rightarrow S/I\rightarrow S/\Sigma_{d-1}\rightarrow0
    	$$
    	we see that
    	\begin{equation}\label{eq:S_rdepth-Skel}
    		S_r\text{-}\!\depth(S/\Sigma_{d-1})\ \ge\ \min\{S_{r+1}\text{-}\!\depth(\Sigma_{d-1}/I)-1,\, S_r\text{-}\!\depth(S/I)\},
    	\end{equation}
    	with equality if $S_r\text{-}\!\depth(S/I)<S_{r+1}\text{-}\!\depth(\Sigma_{d-1}/I)-1$.
    	
    	By \cite[Theorem 1.2]{HSZ}, $\Sigma_{d-1}/\Sigma_d=\Sigma_{d-1}/I$ is a Cohen-Macaulay module of dimension $d$. Hence, Example \ref{Ex:(i)-(vi)}(ii) implies that $S_{r+1}\text{-}\!\depth(\Sigma_{d-1}/I)-1=d-1$. It follows that $S_{r}\text{-}\!\depth(S/\Sigma_{d-1})=t$, if $t<d-1$. On the other hand, if
    	$t=d-1$, then $S_{r}\text{-}\!\depth(S/\Sigma_{d-1})\ge d-1$. However, since $\dim(S/\Sigma_{d-1})= d-1$, we obtain again that $S_{r}\text{-}\!\depth(S/\Sigma_{d-1})=d-1=t$.\hfill$\square$\smallskip
    	
    	\textit{Proof of} (ii):  If $S/I$ satisfies Serre's condition $(S_r)$, then Theorem \ref{Thm:Serre} gives $S_{r}\text{-}\!\depth(S/I)=\dim(S/I)=d$. Hence \cite[Theorem 1.2]{HSZ} and inequality (\ref{eq:S_rdepth-Skel}) imply that $S_{r}\text{-}\!\depth(S/\Sigma_{d-1})\ge d-1$. Since, as mentioned before, $\dim(S/\Sigma_{d-1})=d-1$, Proposition \ref{Prop:inequalities}(a) implies that equality holds, which again by Theorem \ref{Thm:Serre} means that $S/\Sigma_{d-1}$ satisfies Serre's condition $(S_r)$, as wanted.\hfill$\square$\smallskip
    	
    	Now, by induction on $d$ we can assume that
    	$$
    	S_r\text{-}\!\depth(S/\Sigma_{d-1})\ =\ \max\{i:\ S/\Sigma_i^{\bf g}(\Sigma_{d-1})\ \textup{satisfies Serre's condition}\ (S_r)\},
    	$$
    	and $S/\Sigma_{j}^{\bf g}(\Sigma_{d-1})$ satisfies Serre's condition $(S_r)$ for all $j\le S_r\text{-}\!\depth(S/\Sigma_{d-1})$.
    	
    	Noticing that
    	$\Sigma_j=\Sigma_j^{\bf g}(\Sigma_{d-1})$ for $j\leq d-1$, the claims (i)-(ii) and the above inductive hypothesis imply the assertions, completing the inductive proof.
    \end{proof}
    
    The argument of the proof also shows that
    \begin{Corollary}
    	Let $I\subset S$ be a monomial ideal. Then
    	$$
    	S_r\text{-}\!\depth(S/\Sigma_{0}^{\bf g}(I))\ \le\ \cdots\ \le\ S_r\text{-}\!\depth(S/\Sigma_{d-1}^{\bf g}(I))\ \le\ S_r\text{-}\!\depth(S/I),
    	$$
    	for all $r\ge1$ and all ${\bf g}\in\NN^n$ with ${\bf g}\ge{\bf deg}(I)$.
    \end{Corollary}
    
    \section{Eventual behaviour of the Serre depth}\label{sec5}
    
    As in the previous section, let $S=K[x_1,\dots,x_n]$ be the standard graded polynomial ring over a field $K$. Our goal is to establish Theorem \ref{ThmF}.
    \begin{Theorem}\label{Thm:S_rDepth-Mon}
    	Let $I\subset S$ be a monomial ideal. Then
    	$$
    	S_r\text{-}\!\depth(S/I^{k+1})\ =\ S_r\text{-}\!\depth(S/I^k),
    	$$
    	for all $r\ge1$ and all $k\gg0$.
    \end{Theorem}

    To prove this result, we need some preparation.\smallskip
    
    For a monomial $u\in S$, we denote by ${\bf deg}(u)$ the \textit{multidegree} of $u$. That is, the vector ${\bf deg}(u)\in\ZZ^n$ whose $i$-th entry is $({\bf deg}(u))(i)=\deg_{x_i}(u)$, for all $i\in[n]$.
    
    Let ${\bf e}_1,\dots,{\bf e}_n$ be the standard basis of $\ZZ^n$. That is ${\bf e}_i(j)=0$ for all $j\ne i$ and ${\bf e}_i(i)=1$ for all $i$. We regard $S$ as a $\ZZ^n$-graded ring, by setting ${\bf deg}(x_i)={\bf e}_i$ for all $i\in[n]$. Monomial ideals are naturally $\ZZ^n$-graded, also called \textit{multigraded}. Hence, the local cohomology modules $\H_\m^j(S/I)$ of a monomial ideal $I\subset S$, as well as their graded Matlis duals, are naturally multigraded.
    
    The following general lemma describes the Krull dimension of a multigraded finitely generated $S$-module. Given ${\bf a}\in\ZZ^n$, we put $\supp({\bf a})=\{i:\ {\bf a}(i)\ne0\}$.
    
    \begin{Lemma}\label{Lem:dim-multi}
    	Let $M$ be a finitely generated multigraded $S$-module. Then
    	$$
    	\dim_S(M)\ =\ \max_{F\subset[n]}\{|F|:\ M_{\bf a}\ne0\ \textit{for infinitely many}\ {\bf a}\in\ZZ^n\ \textit{with}\ \supp({\bf a})=F\}.
    	$$
    \end{Lemma}
    \begin{proof}
    	For a subset $F\subset[n]$, let $P_{F}= (x_i : i \in[n]\setminus F)$ be a monomial prime ideal. Since $\dim(S/P_F)=|F|$ and $M$ is finitely generated and multigraded, we have
    	\begin{equation}\label{eq:dim-multidegree}
    	\dim_S(M)\ =\ \max\{|F|:\ M_{P_F}\ne0\}.
    	\end{equation}
    	
    	We characterize when $M_{P_F}\ne0$. By the criterion for vanishing in a localization,
    	$$
    	M_{P_F}\ne0
    	\quad\Longleftrightarrow\quad
    	\textup{there exists}\, m \in M \text{ such that } m/1\ne0 \text{ in } M_{P_F},
    	$$
    	which is equivalent to $sm\ne0$ for all $s\in S\setminus P_F$. Since $S \setminus P_F$ has a $K$-basis consisting exactly of the monomials whose support is contained in $F$, we may choose $m$ to be homogeneous, say $m \in M_{{\bf a}}$, and the above condition becomes
    	$$
    	{\bf x^u}m\ne0
    	\quad \text{for all}\ {\bf u}\in \NN^n
    	\ \text{with}\ \supp({\bf u})\subset F.
    	$$
    	
    	Assume first that $M_{P_F}\ne0$, and fix such a homogeneous element $m$. For any ${\bf u}$ with $\supp({\bf u})\subset F$ we have ${\bf x^u}m\ne0$, and multidegree ${\bf deg}({\bf x^u}m)={\bf a}+{\bf u}$. Choosing ${\bf u}$ so that ${\bf u}(i) > 0$ for all $i \in F$ and ${\bf u}(i)=0$ for $i \notin F$, and varying ${\bf u}$, we obtain infinitely many multidegrees $\mathbf b$ with $\supp(\mathbf b)=F$ and $M_{\mathbf b}\ne0$.
    	
    	Conversely, suppose that there exist infinitely many multidegrees ${\bf a}\in\ZZ^n$ with $\supp({\bf a})=F$ such that $M_{{\bf a}}\ne0$.
    	Choose a nonzero homogeneous element $m \in M_{{\bf a}}$. For any monomial $s \in S \setminus P_F$ (that is, any monomial whose support is contained $F$), multiplication by $s$ preserves non-zeroness on infinitely many homogeneous components of $M$, hence $s m\ne0$. Thus $m/1\ne0$ in $M_{P_F}$, and therefore $M_{P_F}\ne0$.
    	
    	Hence, we have shown that $M_{P_F}\ne0$ if and only if we have $M_{{\bf a}}\ne0$ for infinitely many ${\bf a}\in\ZZ^n$ with $\supp({\bf a})=F$. Combining this with the formula (\ref{eq:dim-multidegree}) yields the desired equality.
    \end{proof}

    We need a description of the Matlis duals of local cohomology modules of quotients by monomial ideals due to Takayama \cite{T}. To this end, we need more notation.
    
    Given ${\bf a}\in\ZZ^n$, we put $G_{\bf a}=\{i:\ {\bf a}(i)<0\}$ and $H_{\bf a}=\{i:\ {\bf a}(i)>0\}$. Moreover, we define ${\bf a}^+$ and ${\bf a}^-$ by setting
    $$
    {\bf a}^+(j)=\max\{{\bf a}(j),0\}\ \ \ \textup{and}\ \ \ {\bf a}^-(j)=\min\{{\bf a}(j),0\},\quad\textup{for all}\ j\in[n].
    $$
    
    Given a monomial ideal $I\subset S$ and a vector ${\bf a}\in\ZZ^n$, we define the following simplicial complex on $[n]$:
    $$
    \Delta_{\bf a}(I)\ =\ \{H\setminus G_{\bf a}\ :\ G_{\bf a}\subset H\subset[n]\ \text{and}\ {\bf x}^{{\bf a}^+}\notin I_H\}.
    $$
    
    Here $I_H$ is the \textit{monomial localization} of $I$ at $P_{H}=(x_i:i\in[n]\setminus H)$. That is, the monomial ideal obtained from $I$ by substituting the variables $x_i\notin P_{H}$ by $1$.

    By Takayama \cite[Theorem 1]{T} (see, also, Hochster \cite{Hoc}), we have
    $$
    \H_\m^j(S/I)_{\bf a}\ \cong\
    	\widetilde{\H}_{j-|G_{\bf a}|-1}(\Delta_{\bf a}(I);K),\quad\textup{for all}\ {\bf a}\in\ZZ^n.
    $$
    
    Here, ${\bf 1}=\sum_{i\in[n]}{\bf e}_i$, and $\widetilde{\H}_s(\Gamma;K)$ is the $s$-th reduced simplicial homology of a simplicial complex $\Gamma$ computed with respect to a field $K$.
     
    We remark that in \cite[Theorem 1]{T}, Takayama proves the above formula under some extra conditions for {\bf a}. However following the proof of \cite[Theorem 1]{T} one can observe that if any of those conditions on ${\bf a}$ is violated then both $\H_\m^j(S/I)_{\bf a}$ and $\widetilde{\H}_{j-|G_{\bf a}|-1}(\Delta_{\bf a}(I);K)$ are zero. Thus, those conditions are unnecessary.\smallskip
    
    As explained in formula (\ref{eq:gradedPieceMatlis}), we have
    $$
    \D_S(\H_\m^j(S/I))_{\bf a}\ =\ \Hom_K(\H_\m^j(S/I)_{-\bf a},K),\quad\textup{for all}\ {\bf a}\in\ZZ^n.
    $$
    
    Since $\H_\m^j(S/I)_{\bf a}$ is a finite dimensional $K$-vector space, and $G_{-{\bf a}}=H_{\bf a}$ for all ${\bf a}\in\ZZ^n$, we obtain that
    \begin{equation}\label{eq:gradedPieceMatlisMonomial}
    	\D_S(\H_\m^j(S/I))_{\bf a}\ \cong\ \widetilde{\H}_{j-|H_{\bf a}|-1}(\Delta_{-{\bf a}}(I);K),\quad\textup{for all}\ {\bf a}\in\ZZ^n.
    \end{equation}
    
    Finally, we recall a few facts about \textit{Presburger arithmetic}.
    
    Recall that a \textit{Presburger formula} is a boolean formula with variables in $\ZZ$ that can be written using quantifiers ($\exists$ and $\forall$), boolean operations (or, not, and), and integer affine-linear inequalities in the variables
    
    A variable in a Presburger formula is said to be \textit{free} if it is not quantified. A Presburger formula is written $F({\bf u})$ to indicate that ${\bf u}$ is the list of free variables.
    
    A subset $P\subset\ZZ^n$ is said to be \textit{Presburger definable}, or simply a \textit{Presburger set}, if it can be defined via a Presburger formula $F({\bf u})$, that is $P=\{{\bf u}\in\ZZ^n:\ F({\bf u})\}$.
    
    A subset $P\subset\ZZ^n$ is called \textit{linear} if
    $$
    P\ =\ {\bf a}+\NN{\bf b}_1+\dots+\NN{\bf b}_m\ =\ \{{\bf a}+k_1{\bf b}_1+\dots+k_m{\bf b}_m:\ k_1,\dots,k_m\in\NN\},
    $$
    for some ${\bf a},{\bf b}_1,\dots,{\bf b}_m\in\ZZ^n$. It is also possible that $P=\{{\bf a}\}$ is a singleton, in which case ${\bf b}_1=\dots={\bf b}_m={\bf 0}=(0,0,\dots,0)$.
    
    Whereas, $P\subset\ZZ^n$ is called \textit{semilinear} if $P$ is the union of finitely many linear sets.
    	
    The following classical result is due to Eilenberg and Sch\"utzenberger \cite{ES}. For our aims, we quote a version of this result formulated by D'Alessandro, Intrigila and Varricchio (see \cite[Theorem 1]{DIV}).
    
    \begin{Theorem}\label{Thm:Presburger}
    	Let $P\subset\ZZ^n$. Then $P$ is a Presburger set if and only if $P$ is semilinear.
    \end{Theorem}
    
    It follows directly from this theorem that a finite union of Presburger sets is a Presburger set as well.
    
    Given a Presburger set $P\subset\ZZ^n$ and a non-empty subset $F=\{i_1<\dots<i_t\}\subset[n]$, let $\pi_F:\ZZ^n\rightarrow\ZZ^t$ be the map defined by setting $\pi_F({\bf a})=({\bf a}(i_1),\dots,{\bf a}(i_t))$ for all ${\bf a}\in\ZZ^n$. Let $\pi_F(P)=\{\pi_F({\bf a}):\ {\bf a}\in P\}$ be the \textit{projection} of $P$ with respect to $F$. It follows immediately from Theorem \ref{Thm:Presburger} that $\pi_F(P)$ is again a Presburger set, because any projection of a semilinear set is again semilinear.\medskip
    
    We are now ready to prove Theorem \ref{Thm:S_rDepth-Mon}.
    
    \begin{proof}[Proof of Theorem \ref{Thm:S_rDepth-Mon}]
    	We may assume that $I\ne0$, otherwise there is nothing to prove. Taking into account the definition of Serre depth, it is enough to show that for all $j\ge0$ we have
    	\begin{equation}\label{eq:dim-enough1}
    		\dim_{S}\D_{S}(\H_{\m}^j(S/I^{k+1}))\ =\ \dim_{S}\D_{S}(\H_{\m}^j(S/I^{k})),\quad\text{for all}\ k\gg0.
    	\end{equation}
    	
    	Fix $j\ge0$ and $k\ge1$. Given a non-empty subset $F\subset[n]$, we define the set
        $$
        \mathcal{B}_k(F)\ =\ \{{\bf a}\in\ZZ^n\ :\ \supp({\bf a})=F,\ \D_S(\H_\m^j(S/I^k))_{\bf a}\ne0\}.
        $$
        
        Applying formula (\ref{eq:gradedPieceMatlisMonomial}), we can rewrite it as
        \begin{equation}\label{eq:B_k(F)}
        \mathcal{B}_k(F)\ =\ \{{\bf a}\in\ZZ^n\ :\ \supp({\bf a})=F,\ \widetilde{\H}_{j-|H_{\bf a}|-1}(\Delta_{-{\bf a}}(I^k);K)\ne0\}.
        \end{equation}
        
        We claim that there exists an integer $k_F>0$ such that for $k\ge k_F$ the sets $\mathcal{B}_k(F)$ are all finite or all infinite. Once we acquire this claim, we define $k_0=\max_{F\subset[n]}k_F$. Then, for $k\ge k_0$ and any fixed subset $F\subset[n]$, the sets $\mathcal{B}_k(F)$ are all finite or all infinite. Applying Lemma \ref{Lem:dim-multi}, it follows that (\ref{eq:dim-enough1}) holds for all $k\ge k_0$.
        
        Let $F\subset[n]$. It remains to prove that for $k\gg0$, the sets $\mathcal{B}_k(F)$ are all finite or all infinite. To this end, we employ Presburger's arithmetic. We claim that the following is a Presburger set:
        \begin{equation}\label{eq:B(F)}
        \mathcal{B}(F)\ =\ \{(k,{\bf a})\in\ZZ\times\ZZ^{n}\ :\ k\ge1,\ {\bf a}\in\mathcal{B}_k(F)\}.
        \end{equation}
        
        There are only finitely many simplicial complexes on the vertex set $[n]$, and for such a simplicial complex $\Delta$ we have $\widetilde{\H}_i(\Delta;K)=0$ if $i<0$ or $i>n$. Hence
        $$
        \mathcal{C}\ =\ \{(i,\Delta)\ :\ \widetilde{\H}_i(\Delta;K)\ne0\}
        $$
        is a finite set. Combining formula (\ref{eq:B(F)}) with (\ref{eq:B_k(F)}), we see that $(k,{\bf a})\in\mathcal{B}(F)$ if and only if $k\ge1$ and $\widetilde{\H}_{j-|H_{{\bf a}}|-1}(\Delta_{-{\bf a}}(I^k);K)\ne0$. Equivalently, $(k,{\bf a})\in\mathcal{B}(F)$ if and only if $k\ge1$ and there exists a pair $(i,\Delta)\in\mathcal{C}$ such that
        $$
        \Delta_{-{\bf a}}(I^k)=\Delta\quad\text{and}\quad j-|H_{{\bf a}}|-1=i.
        $$
        
        Since $\mathcal{C}$ is finite, it follows that $\mathcal{B}(F)$ is a finite union of
        sets of the form
        $$
        \mathcal{B}_{(i,\Delta)}(F)\ =\ \{(k,{\bf a})\in\ZZ\times\ZZ^n:\ k\ge1,\ \supp({\bf a})=F,\ \Delta_{-{\bf a}}(I^k)=\Delta,\ j-|H_{{\bf a}}|-1=i\},
        $$
        where $(i,\Delta)\in\mathcal{C}$. Hence, it suffices to prove that each $\mathcal{B}_{(i,\Delta)}(F)$ is a Presburger set. The conditions $k\ge1$, $\supp({\bf a})=F$ and $j-|H_{\bf a}|-1=i$ are obviously Presburger. We show that $\Delta_{-{\bf a}}(I^k)=\Delta$ is a Presburger condition, too. Fix $V\subseteq[n]$. By definition,
        $$
        V\in \Delta_{-{\bf a}}(I^k)\ \ \Longleftrightarrow\ \ V\cap G_{-{\bf a}}=\emptyset\ \ \text{and}\ \ \mathbf x^{(-{\bf a})^+}\notin I_{V\cup G_{-{\bf a}}}^k.
        $$
        
        Let $\mathcal G(I)=\{u_1,\dots,u_m\}$ be the minimal monomial generating set of $I$,
        where $u_i=\mathbf x^{{\bf u}_i}$ for all $i\in[m]$.
        Since $I^k$ is generated by the monomials of the form
        $$
        \prod_{i\in[m]} u_i^{k_i}\ =\ {\bf x}^{\sum_{i\in[m]}k_i{\bf u}_i},
        \quad\text{with}\
        k_i\in\mathbb N\ \textup{such that}\  \sum_{i=1}^m k_i=k,
        $$
        the condition $\mathbf x^{(-{\bf a})^+}\notin I_{V\cup G_{-{\bf a}}}^k$ is equivalent to the following condition: for every $(k_1,\dots,k_m)\in\mathbb N^m$ with $\sum_{i=1}^m k_i=k$, there exists $j\in[n]\setminus(V\cup G_{-{\bf a}})$ such that
        $$
        (-{\bf a})^+(j)\ <\ \sum_{i=1}^m k_i\cdot\deg_{x_j}(u_i)\ =\ \sum_{i=1}^m k_i\,{\bf u}_i(j).
        $$
        
        Equivalently,
        \begin{equation}\label{eq:Pres}
        	\forall k_1,\dots,k_m\!\in\!\mathbb N
        	\Bigl(
        	\sum_{i=1}^m k_i=k\Rightarrow
        	\exists j\in[n]\setminus(V\cup G_{-{\bf a}}):
        	\sum_{i=1}^m k_i\,{\bf u}_i(j)\!-\!(-\!{\bf a})^+(j)\ge 1
        	\Bigr).
        \end{equation}
        
        This is a Presburger condition, as it involves only linear equalities and
        inequalities with integer coefficients, universal quantification over
        $\mathbb N^m$ subject to a linear constraint, and bounded existential
        quantification over the finite set $[n]\setminus(V\cup G_{-{\bf a}})$.
        
        Since the condition $V\cap G_{-{\bf a}}=\emptyset$ is a quantifier-free Presburger condition, membership of a fixed $V$ in $\Delta_{-{\bf a}}(I^k)$ is a Presburger condition.
        
        Now, both $\Delta_{-{\bf a}}(I^k)$ and $\Delta$ are subsets of the power set $\mathcal P([n])$ of $[n]$, which is finite. Therefore, the equality $\Delta_{-{\bf a}}(I^k)=\Delta$ is equivalent to a finite conjunction of Presburger conditions of the form (\ref{eq:Pres}) and their negations, depending on whether $V$ belongs or not to $\Delta$. Thus, $\Delta_{-{\bf a}}(I^k)=\Delta$ is a Presburger condition.
        
        Hence $\mathcal{B}_{(i,\Delta)}(F)$ is a Presburger set for all $(i,\Delta)\in\mathcal{C}$. Since $\mathcal{C}$ is a finite set, it follows that $\mathcal{B}(F)=\bigcup_{(i,\Delta)\in\mathcal{C}}\mathcal{B}_{(i,\Delta)}(F)$ is a Presburger set. Applying Theorem \ref{Thm:Presburger}, we see that $\mathcal{B}(F)$ is a semilinear set. From the definition of semilinear set, it follows immediately that for there exists $k_F>0$ such that the projections $\mathcal{B}_k(F)$, with $k\ge k_F$, are all finite or all infinite, as claimed. This concludes the proof.
    \end{proof}

    Theorem \ref{Thm:S_rDepth-Mon} has the following consequence.
    \begin{Corollary}
    	Let $I\subset S$ be a monomial ideal. Then $S_r\text{-}\!\depth(I^{k+1})=S_r\text{-}\!\depth(I^k)$, for all $r\ge1$ and all $k\gg0$.
    \end{Corollary}
    \begin{proof}
    	We may assume that $I\ne0$. It suffices to show that for all $j\ge0$
    	\begin{equation}\label{eq:dim-enough2}
    		\dim_{S}\D_{S}(\H_{\m}^j(I^{k+1}))\ =\ \dim_{S}\D_{S}(\H_{\m}^j(I^{k})),\quad\text{for all}\ k\gg0.
    	\end{equation}
    	
    	By equation (\ref{eq:iso-R/I-I}), which is valid in our setting because $S$ is Cohen-Macaulay of positive dimension, for all $k\ge1$ we have $\H_\m^0(I^k)=0$, $\H_\m^j(S/I^k)\cong\H_\m^{j+1}(I)$ for all $j=0,\dots,n-2$, and the short exact sequence
    	$$
    	0\rightarrow\H_\m^{n-1}(S/I^k)\rightarrow\H_\m^{n}(I^k)\rightarrow\H_\m^n(S)\rightarrow0.
    	$$
    	Here, we used that $\dim(S/I^k)\le n-1$ and so $\H_\m^n(S/I^k)=0$ (see \cite[Theorem 6.1.2]{BS}). Applying the graded Matlis dual, we obtain the short exact sequence
    	$$
    	0\rightarrow\D_S(\H_\m^n(S))\rightarrow\D_S(\H_\m^{n}(I^k))\rightarrow\D_S(\H_\m^{n-1}(S/I^k))\rightarrow0.
    	$$
    	
    	Altogether, we have
    	$$
    	\dim_S\D_S(\H_\m^j(I^k))=\begin{cases}
    		-\infty&j=0,\\
    		\dim_S\D_S(\H_\m^{j-1}(S/I^k))&0<j<n,\\
    		\max\{\dim_S\D_S(\H_\m^n(S)),\dim_S\D_S(\H_\m^{n-1}(S/I^k))\}&j=n.
    	\end{cases}
    	$$
    	By the proof of Theorem \ref{Thm:S_rDepth-Mon} we have $\dim_S\D_S(\H_\m^i(S/I^{k+1}))=\dim_S\D_S(\H_\m^i(S/I^k))$ for all $i\ge0$ and all $k\gg0$. This fact and the above formula for $\dim_S\D_S(\H_\m^j(I^k))$ imply that equation (\ref{eq:dim-enough2}) indeed holds. This concludes the proof.
    \end{proof}
    
    As explained in the introduction, to prove Conjecture \ref{ConjG} in general, it would be enough to establish equation (\ref{eq:dim_R-const}). If $R$ is complete, this would follow if equation (\ref{eq:ass_R-const}) was true in general. Unfortunately, this is not the case.
    
    \begin{Example}\label{Ex:notAss}
    	\rm Let $S=K[x,y,s,t,u,v]$ be the polynomial ring over a field $K$, $\m=(x,y,s,t,u,v)$, and let $f=sx^2v^2-(s+t)xyuv+ty^2v^2$. Put $R=(S/(f))_\m$. Then $R$ is a Gorenstein local ring with $\dim(R)=5$. Let $I=(u,v)R\subset R$. Katzman (\cite[Corollary 1.3]{Kat}) has proved that $\Ass_R(\H_{I}^2(R))$ is infinite. As a consequence of this fact, in \cite[Theorem 3.2(i)]{VHHK} it is noted that $\Ass_R(\Ext^2_R(R/I^k,R))$ never stabilizes. Since $\omega_R\cong R$, Grothendieck's local duality (\cite[Theorem 3.5.8]{BH}) together with partial Matlis duality (see \cite[Theorem 10.2.19(ii)]{BS}) yield that
    	$$
    	\D_R(\H_\m^3(R/I^k))\ \cong\ \Ext^{\dim(R)-3}_R(R/I^k,\omega_R)\otimes_R\widehat{R}\ \cong\ \Ext^{2}_R(R/I^k,R)\otimes_R\widehat{R},
    	$$
    	for all $k\ge1$. This equation and \cite[Theorem 23.2(ii)]{Mat} imply that
    	$$
    	\Ass_{R}(\D_R(\H_\m^3(R/I^k)))\ =\ \Ass_R(\Ext^{2}_R(R/I^k,R)).
    	$$
    	
    	Hence equation (\ref{eq:ass_R-const}) does not hold in this case. Nonetheless, using Macaulay2 \cite{GDS}, we could see that $\dim_R\D_R(\H_\m^3(R/I^k))=\dim_R\Ext^2_R(R/I^k,R)=2$ for all $2\le k\le 20$. Hence, this lead us to expect that equation (\ref{eq:dim_R-const}) holds anyway in this case.
    \end{Example}

    The following observation is needed for the sequel.
    \begin{Proposition}\label{Prop:depth0}
    	Let $(R,\m)$ be a Noetherian local ring or a standard graded $K$-algebra. Let $M$ be a finitely generated $R$-module, which we assume is homogeneous if $R$ is a $K$-algebra. Then $S_r\text{-}\!\depth(M)=0$ if and only if $\depth(M)=0$. 
    \end{Proposition}
    \begin{proof}
    	If $S_r\text{-}\!\depth(M)=0$, then Proposition \ref{Prop:inequalities} implies that $\depth(M)=0$.
    	
    	Conversely, let $\depth(M)=0$. Then \cite[Theorem 3.5.7]{BH} implies that $\H_\m^0(M)\ne0$. Hence $\dim\D(\H_\m^0(M))\ge0-r+1$ an so $S_r\text{-}\!\depth(M)=0$ for all $r\ge1$.
    \end{proof}

    On the other hand, in support of Conjecture \ref{ConjG}, we can prove the following
    \begin{Theorem}\label{Thm:dim<=3}
    	Let $(R,\m)$ be a Noetherian local ring or a standard graded $K$-algebra. Assume that $R$ is a Cohen-Macaulay domain having a $($graded, if $R$ is a $K$-algebra$)$ canonical module $\omega_R$. Let $I\subset R$ be an ideal, which we assume is homogeneous if $R$ is a $K$-algebra. Suppose that $\dim(R)\le3$. Then
    	$$
    	S_r\text{-}\!\depth(R/I^{k+1})\ =\ S_r\text{-}\!\depth(R/I^k),
    	$$
    	for all $r\ge1$ and all $k\gg0$.
    \end{Theorem}
    \begin{proof}
    	We assume that $I\ne0$, otherwise there is nothing to prove. By assumption $\dim R=d\le3$. Since $R$ is a domain, $\dim(R/I^k)\le 2$ for all $k$. By Brodmann \cite{B79}, $\dim(R/I^{k+1})=\dim(R/I^k)$ is constant for $k\gg0$. By Proposition \ref{Prop:inequalities}, $S_r\text{-}\!\depth(R/I^k)=\depth(R/I^k)$ for all $k\gg0$ and all $r\ge2$. Again by Brodmann \cite{B79a}, it follows that $S_r\text{-}\!\depth(R/I^k)$ is constant for all $k\gg0$ and all $r\ge2$. Therefore, it remains to prove that $S_1\text{-}\!\depth(R/I^k)$ is constant for all $k\gg0$.
    	
    	If $\depth(R/I^k)=0$ for all $k\gg0$, Proposition \ref{Prop:depth0} implies that $S_1\text{-}\!\depth(R/I^k)=0$ for all $k\ge1$. So we may assume that $\depth(R/I^k)\ge1$ for all $k\gg0$.  Since $\dim(R/I^k)\le 2$ for all $k\ge1$, Proposition \ref{Prop:inequalities}(a) gives either $\depth(R/I^k)=1$ for all $k\gg0$ or $\depth(R/I^k)=2$ for all $k\gg0$. If $\depth(R/I^k)=2$ for all $k\gg0$, then $S_1\text{-}\!\depth(R/I^k)=2$ for all $k\gg0$ as well. So, we may therefore assume there exists $k_0>0$ such that $\depth(R/I^k)=1$ for all $k\ge k_0$. Hence $\H_\m^0(R/I^k)=0$ and $\H_\m^1(R/I^k)\ne0$ for all $k\ge k_0$. Taking this into account and applying Grothendieck's local duality, it follows that
    	\begin{equation}\label{eq:S1-dim<=3}
    		S_1\text{-}\!\depth(R/I^k)\ =\ \min\{j\ge 1:\  \dim\Ext^{d-j}_R(R/I^k,\omega_R)\ge j\},
    	\end{equation}
    	for all $k\ge k_0$. Note that $\Ext^0_R(R/I^k,\omega_R)=\Hom_R(R/I^k,\omega_R)\cong(0:_{\omega_R}I^k)$ for all $k\ge1$. Since $R$ is Noetherian and $\omega_R$ is finitely generated, the ascending chain
    	$$
    	(0:_{\omega_R}I)\ \subset\ (0:_{\omega_R}I^2)\ \subset\ (0:_{\omega_R}I^3)\ \subset\ \cdots
    	$$
    	stabilizes. Hence
    	$$
    	\Ass_R(\Ext^0_R(R/I^{k+1},\omega_R))\ =\ \Ass_R(\Ext^0_R(R/I^k,\omega_R)),
    	$$
    	for all $k\gg0$. Moreover, by \cite[Corollary 2.3]{KS} is known that for a finitely generated $R$-module $M$ the sets $\Ass_R(\Ext^{1}_R(R/I^k,M))$ are constant for all $k\gg0$. Applying the previous discussion and this result with $M=\omega_R$, we deduce that
    	\begin{equation}\label{eq:dim-k-k+1}
    		\dim\Ext^{i}_R(R/I^{k+1},\omega_R)\ =\ \dim\Ext^{i}_R(R/I^k,\omega_R),\ \ \text{for}\ i=0,1\ \ \text{and all}\ k\gg0. 
    	\end{equation}
    	
    	We claim that (\ref{eq:dim-k-k+1}) holds also for $i=2$ and all $k\gg0$. Let $k\ge k_0$. The short exact sequence $0\rightarrow I^{k}/I^{k+1}\rightarrow R/I^{k+1}\rightarrow R/I^k\rightarrow0$ induces the exact sequence
    	$$
    	\Ext^2_R(R/I^k,\omega_R)\xrightarrow{\varphi_k}\Ext^2_R(R/I^{k+1},\omega_R)\xrightarrow{\psi_k}\Ext^3_R(I^{k}/I^{k+1},\omega_R).
    	$$
    	By \cite[Corollary 3.5.11]{BH}, $\dim\Ext^3_R(I^{k}/I^{k+1},\omega_R)\le0$. Hence, this dimension is $0$ if the module is non-zero; otherwise it is $-\infty$. We have
    	$$
    	\CoKer(\varphi_k)\ \cong\ \frac{\Ext^2_R(R/I^{k+1},\omega_R)}{\Ker(\psi_k)}\ \cong\ \Im(\psi_k)\ \subset\ \Ext^3_R(I^{k}/I^{k+1},\omega_R).
    	$$
    	It follows that $\dim\CoKer(\varphi_k)\le0$. The short exact sequence
    	$$
    	0\rightarrow\Im(\varphi_k)\rightarrow \Ext^2_R(R/I^{k+1},\omega_R)\rightarrow\CoKer(\varphi_k)\rightarrow0
    	$$
    	implies that
    	$$
    	\dim\Ext^2_R(R/I^{k+1},\omega_R)\ =\ \max\{\dim\Im(\varphi_k),\,\dim\CoKer(\varphi_k)\}.
    	$$
    	On the other hand, we have $\Im(\varphi_k)\cong\Ext^{2}_R(R/I^k,\omega_R)/\Ker(\varphi_k)$. Since $k\ge k_0$, we have $\H_\m^1(R/I^k)\ne0$ and $\Ext^2_R(R/I^k,\omega_R)\ne0$. So, $\dim\Im(\varphi_k)\le\dim\Ext^{2}_R(R/I^k,\omega_R)$ and $\dim\CoKer(\varphi_k)\le0\le\dim\Ext^2_R(R/I^k,\omega_R)$. These facts and the previous inequality imply that
    	$$
    	\dim\Ext^{2}_R(R/I^{k+1},\omega_R)\ \le\ \dim\Ext^{2}_R(R/I^k,\omega_R),\ \ \text{for all}\ k\ge k_0.
    	$$
    	From this it follows immediately that
    	\begin{equation}\label{eq:dim-k-k+2}
    		\dim\Ext^{2}_R(R/I^{k+1},\omega_R)\ =\ \dim\Ext^{2}_R(R/I^k,\omega_R),\ \ \text{for all}\ k\gg0. 
    	\end{equation}
    	Combining equations (\ref{eq:S1-dim<=3}), (\ref{eq:dim-k-k+1}) and (\ref{eq:dim-k-k+2}), the assertion follows.
    \end{proof}

    We can also prove the following result supporting Conjecture \ref{ConjG}.
    \begin{Proposition}\label{Prop:S_1-depth-dim}
    	Let $(R,\m)$ be a Noetherian local ring or a standard graded $K$-algebra. Assume that $R$ is a homomorphic image of a Gorenstein ring. Let $I\subset R$ be an ideal, which we assume is homogeneous if $R$ is a $K$-algebra. Suppose that $\depth(R/I^k)=\dim(R/I^k)-1$ for all $k\gg0$. Then
    	$$
    	S_1\text{-}\!\depth(R/I^{k+1})\ =\ S_1\text{-}\!\depth(R/I^k),
    	$$
    	for all $k\gg0$.
    \end{Proposition}
    \begin{proof}
    	By Brodmann \cite{B79}, there exists $k_0>0$ such that $\Ass(R/I^k)=\Ass(R/I^{k_0})$ for all $k\ge k_0$. Hence, for all $k\ge k_0$, either all rings $R/I^{k}$ are equidimensional and satisfy $(S_1)$ (that is, they do not have embedded associated primes), or all rings $R/I^k$ fail at least one of these properties. Since $R$ is a homomorphic image of a Gorenstein ring, Theorem \ref{Thm:Serre} implies that either $S_1\text{-}\!\depth(R/I^k)=\dim(R/I^k)$ for all $k\ge k_0$, or $S_1\text{-}\!\depth(R/I^k)<\dim(R/I^k)$ for all $k\ge k_0$.
    	
    	By Brodmann's results \cite{B79,B79a}, the functions $k\mapsto\depth(R/I^k)$, $k\mapsto\dim(R/I^k)$ are eventually constant. By assumption, $\depth(R/I^k)=\dim(R/I^k)-1$ for all $k\gg0$. Proposition \ref{Prop:inequalities}(a) implies that $S_1\text{-}\!\depth(R/I^k)\in\{\dim(R/I^k)-1,\dim(R/I^k)\}$ for all $k\gg0$. Combining this with the first part of the argument, it follows that either $S_1\text{-}\!\depth(R/I^k)=\dim(R/I^k)$ for all $k\gg0$, or else $S_1\text{-}\!\depth(R/I^k)=\dim(R/I^k)-1$ for all $k\gg0$. The assertion follows.
    \end{proof}

    Summarizing what we know thus far, Theorems \ref{Thm:S_rDepth-Mon} and \ref{Thm:dim<=3}, Examples \ref{Ex:(i)-(vi)}(ii)-(iii) together with Brodmann classical result \cite{B79a}, and Propositions \ref{Prop:depth0} and \ref{Prop:S_1-depth-dim} imply:
    \begin{Corollary}\label{Cor:resume}
    	Conjecture \ref{ConjG} holds true in the following cases.
    	\begin{enumerate}
    		\item[\textup{(a)}] $R$ is a polynomial ring over a field $K$ and $I$ is a monomial ideal.
    		\item[\textup{(b)}] $R$ is Cohen-Macaulay domain with a $($graded, if $R$ is a $K$-algebra$)$ canonical module $\omega_R$ and $\dim(R)\le3$.
    		\item[\textup{(c)}] $R/I^k$ is $($sequentially$)$ Cohen-Macaulay for all $k\gg0$.
    		\item[\textup{(d)}] $\depth(R/I^k)=0$ for all $k\gg0$.
    		\item[\textup{(e)}] $R$ is a homomorphic image of a Gorenstein ring, $\depth(R/I^k)\!=\!\dim(R/I^k)\!-\!1$ for all $k\gg0$ and $r=1$.
    	\end{enumerate}
    \end{Corollary}
    
    The following concept we discuss, called the \textit{depth strata}, can be defined for any ideal $I$ in a Noetherian local ring or standard graded $K$-algebra, in which case we assume $I$ is homogeneous. However, in the following discussions we will restrict our attention only to monomial ideals.\smallskip
    
    Let $I\subset S$ be a monomial ideal. Note that $\dim(S/I^k)=\dim(S/I)$ for all $k\ge1$. Let $d=\max\{1,\dim(S/I)\}$. By Theorem \ref{Thm:S_rDepth-Mon}, $S_r\text{-}\!\depth(R/I^{k+1})=S_r\text{-}\!\depth(R/I^k)$ for all $k\gg0$. Moreover, Propositions \ref{Prop:inequalities} and \ref{Prop:depth0} imply that
    $$
    S_1\text{-}\!\depth(R/I^k)\ \ge\ \cdots\ \ge\ S_d\text{-}\!\depth(R/I^k),
    $$
    and if $S_i\text{-}\!\depth(R/I^k)=0$ then $S_j\text{-}\!\depth(R/I^k)=0$ for all $j\ge1$.\smallskip
    
    We call the numerical function $f:\ZZ_{>0}\rightarrow\NN^d$ defined by
    $$
    f(k)\ =\ (S_1\text{-}\!\depth(S/I^k),\dots,S_d\text{-}\!\depth(S/I^k))
    $$
    the \textit{depth strata} of $I$.
    
    Let $f(k)=(a_{1,k},\dots,a_{d,k})$ for all $k\ge1$. Summarizing our discussion thus far, the function $f:\ZZ_{>0}\rightarrow\NN^d$ has the following three properties:
    \begin{enumerate}
    	\item[(i)] $a_{1,k}\ge\dots\ge a_{d,k}\ge0$ for all $k\ge1$,
    	\item[(ii)] if for some $k\ge1$, we have $a_{i,k}=0$, then $a_{j,k}=0$ for all $j\in[d]$, and
    	\item[(iii)] $a_{i,k+1}=a_{i,k}$ for all $i\in[n]$ and all $k\gg0$.
    \end{enumerate}
    
    In view of these facts, we ask the following
    \begin{Question}\label{Ques:strata}
    	Let $f:\ZZ_{>0}\rightarrow\NN^d$ be a numerical function satisfying the properties \textup{(i)-(ii)-(iii)}. Does there exist a monomial ideal $I$ in some polynomial ring $S$ whose depth strata coincides with $f$ ?
    \end{Question}

    In \cite{HNTT}, the authors proved that any convergent non-negative numerical function is actually the depth function of a monomial ideal, not merely a homogeneous ideal. We expect that Question \ref{Ques:strata} has a positive answer, even though at the moment it looks very hard to construct monomial ideals whose depth strata has strange behaviour.\bigskip
    
	\textbf{Acknowledgments.} A. Ficarra was supported by the Grant JDC2023-051705-I funded by MICIU/AEI/10.13039/501100011033 and by the FSE+ and also by INDAM (Istituto Nazionale Di Alta Matematica).

\bigskip
	
	\begin{thebibliography}{99}
		
		\bibitem{B79} M. Brodmann, \textit{Asymptotic stability of $\textup{Ass}(M/I^nM)$}, Proc. Am. Math. Soc., 74(1979), 16--18
		
		\bibitem{B79a} M. Brodmann, \textit{The asymptotic nature of the analytic spread}, Math. Proc. Cambridge Philos. Soc., {\bf 86} (1979), 35--39.
		
		\bibitem{BS} M. Brodmann, R. Sharp, {\it Local Cohomology: An Algebraic Introduction with Geometric Applications}, Cambridge Studies in Advanced Mathematics 136, 2nd Edition, Cambridge University Press, 2013.
		
		\bibitem{BH} W. Bruns and  J. Herzog, \textit{Cohen-Macaulay Rings}, revised ed., Cambridge Stud. Adv. Math., 39, Cambridge University Press, Cambridge, 1998
		
		\bibitem{Conca} A. Conca, \textit{A note on the v-invariant}, Proc. Amer. Math. Soc., 2024, 152(6), pp. 2349--2351
		
		\bibitem{CV}  A. Conca, M. Varbaro, \textit{Square-free Gröbner degenerations},  Invent. Math. {\bf 221} (2020), no. 3, 713--730.
		
		\bibitem{CHT} S.D. Cutkosky, J. Herzog, N.V. Trung. {\it Asymptotic Behaviour of the Castelnuovo--Mumford regularity}. Compositio Mathematica {\bf 118}, (1999), 243--261.
		
		\bibitem{DIV} F. D'Alessandro, B. Intrigila, S. Varricchio, \textit{Quasipolynomials, linear Diophantine equations, and semi-linear sets}, Theoret. Comput. Sci. 416 (2012), 1--16.
	    
	    \bibitem{DMV} H. Dao, L. Ma, M. Varbaro, \textit{Regularity, singularities and h-vectors of graded algebras},
	    Trans. Amer. Math. Soc. 377 (2024), no. 3, 2149--2167.
	    
	    \bibitem{DMMNW} H. Dao, E. Miller, J. Monta\~{n}o, C. O'Neill, K. Woods, \textit{Quasipolynomial behavior via constructibility in multigraded algebra}, (2025), preprint \url{https://arxiv.org/abs/2512.18536}.
	    
	    \bibitem{DM} H. Dao, J. Monta\~{n}o, \textit{Length of local cohomology of powers of ideals}, Trans. Amer. Math. Soc. 371 (2019), 3483--3503.
	    
		\bibitem{ES} S. Eilenberg, M.-P. Sch\"utzenberger, \textit{Rational sets in commutative monoids}, J. Algebra 13 (1969) 173--191.
		
		\bibitem{FS} A. Ficarra, E. Sgroi, \textit{Asymptotic behaviour of the $\v$-number of homogeneous ideals}, 2023, preprint \url{https://arxiv.org/abs/2306.14243}
		
		\bibitem{GDS} D.~R.~Grayson, M.~E.~Stillman. {\em Macaulay2, a software system for research in algebraic geometry}. Available at \url{http://www.math.uiuc.edu/Macaulay2}.
		
		\bibitem{EGAIV65} A.\ Grothendieck (with J.\ Dieudonn\'e), \emph{\'El\'ements de g\'eom\'etrie alg\'ebrique, IV.} \'Etude locale des sch\'emas et des morphismes de sch\'emas. (Seconde partie), Publ. I.H.E.S., No. {\bf 24}, 1965.
		
		\bibitem{HNTT} H.T. H\`a, H. Nguyen, N. Trung, T. Trung, \textit{Depth functions of powers of homogeneous ideals}, Proc. AMS, {\bf 149} (2021), 1837--1844.
		
		\bibitem{HRW} W. Heinzer, C. Rotthaus, S. Wiegand, \textit{Integral Domains Inside Noetherian Power Series Rings. Constructions and Examples}. Mathematical Surveys and Monographs 259, American Mathematical Society (AMS), Providence (2021)
		
		\bibitem{HH} J. Herzog, T. Hibi, \textit{The depth of powers of an ideal}, J. Algebra 291 (2005), no. 2, 534--550.
		
		\bibitem{HHBook} J.~Herzog, T.~Hibi, \emph{Monomial ideals}, Graduate texts in Mathematics {\bf 260}, Springer, 2011.
		
		\bibitem{HS} J. Herzog, E. Sbarra, \textit{Sequentially Cohen-Macaulay modules and local cohomology}, in Algebra, arithmetic and geometry, Part I, II (Mumbai, 2000), 327–340, Tata Inst. Fund. Res. Stud. Math., 16, Tata Inst. Fund. Res., Bombay.
		
		\bibitem{HSZ} J. Herzog, A. Soleyman Jahan, X. Zheng, \textit{Skeletons of monomial ideals}. Math. Nachr. 283 (2010), no. 10, 1403--1408.
		
		\bibitem{Hibi} T. Hibi, \textit{Quotient Algebras of Stanley–Reisner Rings and Local Cohomology}, J. Alg. bf 140, (1991), 336--343.
		
		\bibitem{Hoc} M. Hochster, \textit{Cohen-Macaulay rings, combinatorics, and simplicial complexes}, in Ring theory, II (Proc. Second Conf., Univ. Oklahoma, Norman, Okla., 1975), pp. 171--223, Lect. Notes Pure Appl. Math., Vol. 26, Dekker, New York-Basel, 1977.
		
		\bibitem{KMT} T. Kataoka, Y. Muta, N. Terai, \textit{The $\v$-numbers of Stanley–Reisner ideals from the viewpoint of Alexander dual complexes}, J. Algebra 684 (2025), 589–611.
		
		\bibitem{Kat} M. Katzman, \textit{An example of an infinite set of associated primes of a local cohomology module}, J. Algebra, 252 (2002), 161-166.
		
		\bibitem{KS} K. Khashyarmanesh, Sh. Salarian, \textit{Asymptotic stability of $\Att_R\Tor^R_1(R/\mathfrak{a}^n,A)$} Proc. Edinb. Math. Soc. (2) 44(3):479--483, (2001).
		
		\bibitem{Kod} V. Kodiyalam. {\it Asymptotic Behaviour of Castelnuovo-Mumford regularity}. Proc. Amer. Math. Soc. {\bf 198}(2), (1999), 407--411.
		
		\bibitem{MV} T. Marley, J. C. Vassilev, \textit{Cofiniteness and associated primes of local cohomology modules}, J. Algebra, 256 (2002), 180–193.
		
		\bibitem{Mat} H. Matsumura. \textit{Commutative Ring Theory}, Cambridge University Press, 1989.
		
		\bibitem{MT} Y. Muta, N. Terai, \textit{The Serre depth of Stanley-Reisner rings and the depth of their symbolic powers}, (2025), preprint \url{https://arxiv.org/abs/2509.14838}.
		
		\bibitem{PPTYa} M.R. Pournaki, M. Poursoltani, N. Terai, S. Yassemi, \textit{Simplicial complexes satisfying Serre’s condition versus the ones which are Cohen–Macaulay in a fixed codimension}, SIAM J. Discrete Math. 36 (2022), no. 4, 2506–2522
		
		\bibitem{PPTYb} M.R. Pournaki, M. Poursoltani, N. Terai, S. Yassemi, \textit{On the dimension of dual modules of local cohomology and the Serre’s condition for the unmixed Stanley–Reisner ideals of small height}, J. Algebra 632 (2023), 751–782.
		
		\bibitem{Rot} J. Rotman, \textit{Advanced Modern Algebra}, American Mathematical Soc., (Vol 114), 2010.
		
		\bibitem{Sb} E. Sbarra, \textit{Upper bounds for local cohomology for rings with given Hilbert function}, Comm. Algebra 29 (2001), no. 12, 5383--5409.
		
		\bibitem{Sc0} P. Schenzel, \textit{Zur lokalen Kohomologie des kanonischen Moduls}, Math. Z. 165, 223--230, 1979.
		
		\bibitem{Sc1} P. Schenzel, \textit{Dualisierende Komplexe in der lokalen Algebra und Buchsbaum-Ringe}, Lecture Notes in Math., vol. 907, Springer, Berlin, 1982.
		
		\bibitem{Sc2} P. Schenzel, \textit{On the use of local cohomology in algebra and geometry}. Six lectures on commutative algebra. Basel: Birkhäuser Basel, 1998. 241-292.
		
		\bibitem{Sc} P. Schenzel, \textit{On the dimension filtration and Cohen-Macaulay filtered modules}. In Commutative algebra and algebraic geometry (Ferrara), volume 206 of Lecture Notes in Pure and Appl. Math., pages 245–264. Dekker, New York, 1999.
		
		\bibitem{Smith} D. Smith, \textit{On the Cohen-Macaulay property in commutative algebra and simplicial topology}, Pacific J. Math. 141, (1990), 165--196.
		
		\bibitem{St} R. P. Stanley, \textit{Combinatorics and Commutative Algebra}, Birkh\"auser, 1983.
		
		\bibitem{T} Y. Takayama, \textit{Combinatorial characterizations of generalized Cohen-Macaulay monomial ideals}, Bull. Math. Soc. Sci. Math. Roumanie (N.S.), 48 (2005), 327--344.
		
		\bibitem{VHHK} N. Van Hoang, P. Huu Khanh, \textit{On the asymptotic stability of certain sets of prime ideals}, East West Math 14.1 (2012).
		
		\bibitem{V} W. Vasconcelos, \textit{Integral Closure: Rees Algebras, Multiplicities, Algorithms}, Springer Monographs in Mathematics, Springer-Verlag, Berlin (2005).
	\end{thebibliography}
\end{document}